\documentclass[onecolumn]{IEEEtran}

\IEEEoverridecommandlockouts                              

\usepackage{ulem}

\usepackage{amsmath,amsfonts,amssymb,color,dsfont}
\usepackage{graphicx}
\usepackage{makeidx} 
\usepackage{amsmath} 
\usepackage{mathrsfs}
\usepackage{amsfonts}
\usepackage{mathtools}
\usepackage{comment}
\usepackage{cite}
\usepackage[hidelinks]{hyperref}
\usepackage{lipsum}
\usepackage{epstopdf}
\usepackage{algorithmic}
\ifpdf
  \DeclareGraphicsExtensions{.eps,.pdf,.png,.jpg}
\else
  \DeclareGraphicsExtensions{.eps}
\fi

\usepackage{subfig}

\newtheorem{thm}{Theorem}[section]
\newtheorem{prop}[thm]{Proposition}
\newtheorem{lem}[thm]{Lemma}
\newtheorem{cor}[thm]{Corollary}
\newtheorem{asp}[thm]{Assumption}
\newtheorem{defn}[thm]{Definition}

\DeclareMathOperator{\diag}{diag}

\newcommand{\cL}{\mathcal L}
\newcommand{\Lp}{\mathbb L}

\newcommand{\filt}{L} 

\newcommand{\Wb}{\bd{W}}

\newcommand{\U}{\bd{U}}
\newcommand{\Pb}{\bd{P}}
\newcommand{\Qb}{\bd{Q}}
\newcommand{\T}{{\bd{T}}}

\newcommand{\G}{\bd{\mathsf{G}}}
\newcommand{\C}{C}

\newcommand{\DR}{\bd{DF}}

\newcommand{\z}{\bd{z}}

\newcommand{\zh}{\bd{\hat{z}}}
\newcommand{\w}{w}

\newcommand{\Ta}{\bd{T}_{\alpha}}
\newcommand{\zhP}{{{\zh}_P}}
\newcommand{\Uzh}{\bd{U}_{\zhP}}

\newcommand{\DUzh}{\Delta \Uzh}

\newcommand{\Fb}{\bd{F}}

\newcommand{\Ukc}{{U_{P}} }

\newcommand{\Ykc}{Y_{-B_cK_{t_f}} }

\makeatletter
\newcommand{\mathleft}{\@fleqntrue\@mathmargin0pt}
\newcommand{\mathcenter}{\@fleqnfalse}
\makeatother

\newcommand{\bd}[1]{\boldsymbol{#1}}

\newenvironment{proof}{\begin{IEEEproof} }{\end{IEEEproof} }

\renewcommand{\ss}{{\mathcal H} }

\renewcommand\bd[1][]{#1}

\title{\LARGE Extended Kalman filter based observer design for semilinear infinite-dimensional systems} 
\author{Sepideh Afshar, Fabian Germ, Kirsten Morris 
\thanks{Dep. of Radiology, Harvard Medical School, Massachusetts General Hospital, Boston, MA, USA (SA), School of Mathematics and Maxwell Institute for Mathematical Sciences, University of Edinburgh, Edinburgh, UK, (FG), Dept. of Applied Mathematics, University of Waterloo, Waterloo, ON, Canada (KM)
        {\tt\small  safshar1@mgh.harvard.edu}, {\tt\small  f.germ@ed.ac.uk},  {\tt\small kmorris@uwaterloo.ca}  }%
\thanks{Financial  support of  Natural Sciences and Engineering Research Council of Canada (NSERC) and of the  U.S. AFOSR under Grant FA9550-16-1-0061 for this research is gratefully acknowledged.  We also  thank Shuxia Tang for her comments on early version of this paper.}
}

\begin{document}

\thispagestyle{empty}
\maketitle

\begin{abstract}

In many physical applications, the system's state varies with spatial variables as well as time. The state of such systems is modelled by partial differential equations and evolves on an infinite-dimensional space. Systems modelled by delay-differential equations are also infinite-dimensional systems.
The full state of these systems cannot be measured. 
 Observer design is an important tool for estimating the state from available measurements. For linear systems, both finite- and infinite-dimensional,  the Kalman filter provides an estimate with minimum-variance on the error, if certain assumptions on the noise are satisfied. The extended Kalman filter  (EKF) is one type of  extension  to nonlinear finite-dimensional systems.
In this paper we provide  an extension of the EKF  to semilinear infinite-dimensional systems. Under mild assumptions we prove the well-posedness  of equations defining the  EKF. Local exponential stability of the error dynamics is shown. Only detectability is assumed, not observability, so this result is new even for finite-dimensional systems.  The results are illustrated with implementation of finite-dimensional approximations of the infinite-dimensional EKF on an example. 
\end{abstract}

\section{Introduction}
In many physical applications, the system's state varies with spatial variables as well as time. The state of such systems is modelled by partial differential equations and evolves on an infinite-dimensional space, and so they are an important class of infinite-dimensional systems. Systems modelled by delay-differential equations are also infinite-dimensional systems.
The full state of these systems cannot be measured. 
 As for finite-dimensional systems, a system, referred to as an observer or estimator,  can  be designed to estimate the state using the  mathematical model and the measurements provided by sensors.

For linear systems, the Kalman filter (KF)   minimizes the variance of the error under certain assumptions on the disturbances.  The observer can be calculated through solution of a Riccati equation. The Kalman filter is widely used and was  extended to infinite-dimensional linear systems in the 1970's;  see  the review papers \cite{Curtain-SIAM-review} and \cite{Bensoussan-article-2003}. This theory was recently extended to time-varying infinite-dimensional systems \cite{WuJacob2015}.
For linear infinite-dimensional systems, there are a number of  other different approaches to observer design in addition to the Kalman filter, including backstepping
 \cite{OB4-2005},  sliding mode combined with backstepping \cite{MirandaMoreno}. Some other approaches can be found in  \cite{OBDPS1-1982,OBDPS4-2004,Sh4b1,Sh4b2}.

Due to its success in a wide range of applications, an extension of the KF to nonlinear systems, the extended Kalman filter (EKF), was developed for finite-dimensional systems. The EKF design is based on  a linear approximation of the system around the estimated state.  The linearized system is used to derive a Riccati equation and this is used to calculate the observer gain; e.g. \cite{simon2006,KF1-2011}.  This method is widely used; see for example, \cite{stab1,1999_reif_stability_ekf,2000_reif_stability_ekf, R1-2010,R2-2011,R3-1999}. 
However, although this method may work well, it is well known that it may lead to divergent error estimates.
Convergence of the estimation error for EKF depends on the size of the nonlinearity and the initial condition, see for instance, \cite{Reif-Liang1983,ribeiro2004}. 

Local asymptotic convergence of the estimation error under an observability assumption  of the nonlinear system has been shown  in \cite{elizabeth2015,alonge2014}. In  \cite{boutayeb1997},  conditions for asymptotic convergence  are imposed on the linearization residues. Under an uniform observability condition it is shown in \cite{1999_reif_stability_ekf,2000_reif_stability_ekf}, that the estimation error is bounded in presence of  disturbances. Local exponential convergence of  the estimation error under  uniform controllability and  detectability conditions is shown in \cite{baras1988}. Local exponential convergence of the error is proven in \cite{stab2}  under certain assumptions that imply observability; similar results for discrete time are in \cite{1999_reif_stability_ekf}. In \cite{AhrensKhalil2007} better convergence was obtained by using the normal form of the governing ordinary differential equations.  Local asymptotic convergence of the estimation error with an assumption of observability  of the nonlinear system was subsequently shown in  \cite{elizabeth2015,alonge2014}.

Observers  for  nonlinear infinite-dimensional systems  are often designed using a finite-dimensional approximation of the system. This enables the use of techniques  for nonlinear finite-dimensional systems. Some examples are the robust fuzzy  and also robust adaptive observers in \cite{OBDPS6-2008} and \cite{OB7-2011}.  In \cite{AMKestn-paper} the effect of approximation on observer performance for several different types of diffusion models and different observer designs was studied. The EKF has been used on finite-dimensional approximations of PDEs; for example a highway traffic model in \cite{Highway_EKF} and   state-of-charge estimation in lithium-ion batteries \cite{2018_morris_EKF_lithium}.

There are some studies for nonlinear infinite-dimensional systems  where the observer is designed  directly using the infinite-dimensional system equations. In \cite{OB6-2012}, a second-order sliding mode observer is employed to provide stability with the assumption that the measurement is available everywhere. 
In \cite{OB3-2002}, the observer dynamics are corrected by a linear output error injection term via an estimated spatially distributed measurement. Spatially-distributed linear output injection is also proposed in \cite{OB26-2005} for a one-dimensional nonlinear Burgers' equation. Backstepping is used in \cite{Sh4b1} to design an observer for a lithium-ion battery model using the PDE model directly. 
An example of a general and abstract form of late lumping nonlinear observer design is introduced in \cite{OB28-2003} on reflexive Banach spaces, where  a nonlinear feedback operator is added to a copy of the system's dynamics.
Other examples of observer design  for specific nonlinear PDEs can  be found in  \cite{OB20-2008,OB25-2013,OB7-2013,OB28-2014}

In this paper, the EKF is formally shown to be well-posed for a class of semilinear infinite-dimensional systems with bounded observation. As for a finite-dimensional EKF, the observer dynamics are a copy of the original system's dynamics with an injection gain defined by the solution of  an Riccati equation.
Since the Riccati equation is coupled with the observer equation, conventional results in the literature including \cite{CRiccati-1976} for existence of solutions to the Riccati equation cannot be directly used. This is due to the fact that for linear equations the Riccati equation  does not depend on the state of the system. In our, nonlinear, case, such a dependence still remains after linearizing the system, making the analysis more involved.    The  proof of well-posedness was done in \cite{Fabian} for  nonlinearities without time dependence, and briefly sketched in the conference paper  \cite{CDC2020}. A  complete proof with a slightly different presentation and considering time-dependent nonlinearities is provided in this paper,
and a more complex example is presented than in \cite{CDC2020}.

It is also shown that for sufficiently small initial error, and smooth nonlinearity, the error dynamics are exponentially  stable.  The approach for finite-dimensional systems cannot be used here because the analogue of the  observability assumption would be uniformly exact observability, which is extremely restrictive  for infinite-dimensional systems. Local exponential stability of the error  dynamics is shown with much weaker assumptions of uniform stabilizability/detectabilty. Thus, these results are new even for finite-dimensional systems. The estimation error  bounded in presence of disturbances. Although an EKF is generally implemented in discrete time, the  analysis throughout is for a continuous time observer in order to remove the effects of time discretization from those of linearization. 

For implementation, the infinite-dimensional EKF must be approximated using  some method.  The paper concludes with illustration of implementation of  this approach for estimation of concentration in a magnetic drug delivery system. 

\section*{Notation}
Throughout the paper,  calligraphic $\mathcal{H}$, with or without indices,  will denote Hilbert spaces.
 Where the space considered is not clear from the context, the norm $\|\cdot\|$, as well as the scalar product $(\cdot,\cdot)$  will be equipped with an appropriate subscript, i.e. $\|\cdot\|_{\mathcal{H}}$ or $(\cdot,\cdot)_{\mathcal{H}}$; otherwise it is omitted. 
 For an arbitrary, but henceforth fixed $t_f>0$ we work on the time interval $[0,t_f]$. For $0<p<\infty$ let $\Lp^p ([0,t_f],\mathcal{H})$ and $\mathbb{W}^{1,p}([0,t_f],\mathcal{H})$ denote the spaces of functions $\bd{f}:[0,t_f]\rightarrow\mathcal{H}$ such that
 $$
 \int_0^{t_f} \|\bd{f}(t)\|^p_{\mathcal{H}}\,ds<\infty\quad\text{and}\quad \int_0^t \|\bd{f}(t)\|^p_{\mathcal{H}}+\|\frac{d}{dt}\bd{f}(t)\|^p_{\mathcal{H}}\,ds< \infty
 $$
respectively. By $\mathcal{H}_2\hookrightarrow\mathcal{H}_1$ we denote a continuous and dense embedding and the trace space of $\Lp^{p}([0,t_f];\mathcal{H}_2)\cap \mathbb{W}^{p}([0,t_f];\mathcal{H}_1)$ is denoted by $\mathbb{H}_{p,1/p}.$
For integers $k\geq 0$ we denote by $\mathcal{C}^k ([0,t_f],\mathcal{H})$ the space of functions $f:[0,t_f]:\rightarrow\mathcal{H}$ that are $k$ times continuously (with respect to the norm in $\mathcal{H}$) differentiable. Whenever $\mathcal{H}=\mathbb{R}$ we may omit the space and simply write $\Lp([0,t_f]),\mathbb{W}^{1,p}([0,t_f])$ and $\mathcal{C}^k([0,t_f])$ for the spaces above. For a linear operator $\bd{A}:\mathcal{H}_1\rightarrow\mathcal{H}_2$ we write $\mathcal{D}(\bd{A})\subset\mathcal{H}_1$ for its domain and denote by 
$$
\|\bd{A}\|=\sup_{\bd{v}\in\mathcal{H}_1:\|\bd{v}\|\leq 1}\|\bd{A}\bd{v}\|_{\mathcal{H}_2}
$$ 
the usual operator norm. For nonlinear operators $\bd{F}:\mathcal{H}_1\rightarrow\mathcal{H}_2$ we use brackets $\bd{F}(\bd{v})$ whenever we write $\bd{F}$ acting on $\bd{v}\in\mathcal{H}_1$. Lastly, by $\mathcal{L}(\mathcal{H}_1,\mathcal{H}_2)$ we mean the space of linear and bounded operators mapping from $\mathcal{H}_1$ to $\mathcal{H}_2$, where we use the convention $\mathcal{L}(\mathcal{H})=\mathcal{L}(\mathcal{H},\mathcal{H}).$

\section{Preliminaries and statement of observer design problem}

Let $\bd{A}:\mathcal{D}(\bd{A})\rightarrow\mathcal{H}$ be a linear operator that generates a $C_0$-semigroup $\bd{T}(t)$ on $\mathcal{H}$ and $\bd{F}:\mathcal{H}\times [0,t_f]\rightarrow\mathcal{H}$ be strongly continuous in time and nonlinear on $\mathcal{H}$ satisfying $\bd{F}(0,t)=0$ for every $t$. 

We consider the semilinear evolution system

\begin{equation}
\label{Sys1}
\begin{aligned}
\frac{\partial \bd{z}(t)}{\partial t}&=\bd{A}\bd{z}(t)+\bd{F}(\bd{z}(t),t)+\bd{Bu}(t) + \bd{G\omega}(t),\\
\bd{z}(0)&=\bd{z}_0\in\mathcal{H},
\end{aligned}
\end{equation}
where $\bd{u}(t)\in \mathcal{C}([0,t_f],\mathcal{H}_1)$ is the control input, $\bd{\omega}(t)\in\mathcal{C}([0,t_f],\mathcal{H}_2)$ is the input disturbance and $\bd{B}\in\mathcal{L}(\mathcal{H}_1,\mathcal{H})$, $\bd{G}\in\mathcal{L}(\mathcal{H}_2,\mathcal{H})$.
We refer to $\bd{z}(t)$ as the state of the system \eqref{Sys1}.\newline

We impose the following regularity on $\bd{F}$, which is henceforth assumed to hold throughout the paper.

\begin{asp}
\label{asp Frechet}
The operator $\bd{F}$ admits a Fr\'{e}chet-derivative $\bd{DF}(\cdot,\cdot)$ that is globally bounded (in operator norm) as well as locally Lipschitz, uniformly in time. More precisely, there exists a constant $\delta_{DF} >0$ such that $\|\bd{DF}(x,t)\|\leq \delta_{DF}$ for all $(x,t)\in\mathcal{H}\times [0,t_f]$, and for every $\delta>0$ there exists a Lipschitz constant $\iota_{DF}>0$ such that for all $\|\bd{x}-\bd{y}\|<\delta$ and all $t\in [0,t_f]$,
$$
\|\bd{DF}(\bd{x},t)-\bd{DF}(\bd{y},t)\|\leq\iota_{DF}\|\bd{x}-\bd{y}\|.
$$
\end{asp}
Though the following is a well known consequence, it is provided as proposition for the sake of completeness.
\begin{prop}
The operator $\bd{F}$ is globally Lipschitz, uniformly in time, meaning that there exists $M>0$ such that for all $\bd{x},\bd{y}\in\mathcal{H}$ and $t\in [0,t_f]$,
$$
\|\bd{F}(\bd{x},t)-\bd{F}(\bd{y},t)\|\leq M \|\bd{x}-\bd{y}\|.
$$
\end{prop}
\begin{proof}
Fix $t\in [0,t_f]$. By the continuity of $\bd{DF}(\cdot,t)$ on $\mathcal{H}$, the Mean Value Theorem \cite[Thm. 5.1.12]{nonlinear_analysis} yields that for all $\bd{x},\bd{y}\in\mathcal{H}$,
\begin{align}
\|\bd{F}(\bd{x},t)-\bd{F}(\bd{y},t)\|&\leq\sup_{\theta\in [0,1]}\|\bd{DF}(\bd{x}+\theta \bd{y},t)\|\|\bd{x}-\bd{y}\|\nonumber\\
&\leq\delta_{DF}\|\bd{x}-\bd{y}\|,
\end{align}
whereby our proposition holds with $M=\delta_{DF}$.
\end{proof}

For convenience, the disturbance $\bd{\omega}$ and control $\bd{u}$  may be lumped as a single input, 
\begin{equation*}
\bd{B}_d=[\bd{B},\bd{G}],\quad \bd{u}_d^T(t)=[\bd{u}^T(t),\bd{\omega}^T(t)].
\end{equation*}
The state-equation for $\bd{z}$ in system \eqref{Sys1} can then be written in the general form 
\begin{equation}
\label{Sys1g}
\begin{aligned}
\frac{\partial \bd{z}(t)}{\partial t}&=\bd{A}\bd{z}(t)+\bd{F}(\bd{z}(t),t)+\bd{B}_d\bd{u}_d(t),\\
z(0)&=z_0\in\mathcal{H}.
\end{aligned}
\end{equation}

It is useful to establish in what sense the systems considered in this paper admit solutions. \newline
Recall that $\bd{T}(t)$ is the $C_0$-semigroup generated by $\bd{A}$.

\begin{defn}
We say $\bd{z}(t)\in\mathcal{C}([0,t_f],\mathcal{H})$ is a mild solution of \eqref{Sys1g} if for $t\in[0,t_f]$ it satisfies the integral equation
\begin{equation}
\label{mild solution}
\bd{z}(t)=\bd{T}(t)\bd{z}_0+\int_0^t\bd{T}(t-s)(\bd{F}(\bd{z}(s),s)+\bd{B}_d\bd{u}(s))\,ds.
\end{equation}
\end{defn}

It is worth noting that if $\bd{z}(t)$ is a classical, that is, continuously differentiable solution to \eqref{Sys1g}, then it clearly satisfies \eqref{mild solution}. However, to obtain a classical solution, one has to at least impose Lipschitz continuity in time on $\bd{F}$. Such conditions are often not met in applications. The systems considered in this paper are of a more general form, and the following result, ensures the existence of their mild solutions.  For the proof of the following result we refer to \cite[Thm. 6.1.12]{pazy}.

\begin{thm}
\label{Thm Pazy}
Consider a system of the form \eqref{Sys1g}, where $\bd{A}$ generates the $C_0$-semigroup $\bd{T}(t)$ on $\mathcal{H}$, the nonlinearity $\bd{F}(x,t)$ satisfies Assumption \ref{asp Frechet} and $\bd{u}(t)\in\mathcal{C}([0,t_f])$. Then \eqref{Sys1g} has a unique mild solution $\bd{z}(t)\in\mathcal{C}([0,t_f],\mathcal{H})$ given by formula \eqref{mild solution}.
\end{thm}

Let the system measurement be 
$$\bd{y}(t)=\bd{C}\bd{z}(t)+\bd{\eta}(t)$$
where $\bd{\eta}(t)\in \mathcal{C}([0,t_f],\mathbb{R}^p)$, $p\geq 1$, is the output disturbance, and $\bd{C}\in\mathcal{L}(\mathcal{H},\mathbb{R}^p)$.

Our objective is to design an observer for the system \eqref{Sys1g}.
Most generally, an observer is a dynamical system with state $\hat{\bd{z}} (t)$ such that, in the absence of disturbances,
$$ \lim_{t \to \infty} \|\bd{z}(t)-\hat{\bd{z}}  (t)\| =0 . $$
In this paper,  as is common, the observer dynamics contain a copy of the system's dynamics and a feedback term that corrects for the error between the predicted observation,  $\bd{C}\hat{\bd{z}}$, and the actual observation, $\bd{y}.$ The general form of the  observer is 
\begin{equation}
\label{OBD}
\begin{aligned}
\frac{\partial \hat{\bd{z}}(t)}{\partial t}&=\bd{A}\hat{\bd{z}}(t)+\bd{F}(\hat{\bd{z}}(t),t)+\bd{Bu}(t)
+\filt (t)[\bd{y}(t)-\bd{C}\hat{\bd{z}}(t)]\\
\hat{\bd{z}}(0) &=\hat{\bd{z}}_0\in\mathcal{H},
\end{aligned}
\end{equation}
where $\filt (t)$, referred to as observer gain, needs to be selected  so that  in the absence of disturbances $\bd{\omega}(t)$ and $\bd{\eta}(t)$, $ \hat{\bd{z}} (t) \to \bd{z}(t).$ 
The following proposition follows from Theorem \ref{Thm Pazy} and ensures the existence of a mild solution $\hat{\bd{z}}$ to \eqref{OBD} if $\bd{K}$ is strongly continuous.
\begin{prop}
\label{WPness}
Let the assumptions of Theorem \ref{Thm Pazy} hold and let again $\bd{z}(t)$ be the mild solution to \eqref{Sys1g}. Let moreover $\bd{L}\in\mathcal{C}([0,t_f],\mathcal{L}(\mathbb{R}^p,\mathcal{H}))$. Then there exists a unique mild solution $\hat{\bd{z}}(t)$ to \eqref{OBD} given by
\begin{equation*}
\hat{\bd{z}}(t)=\bd{T}(t)\hat{z}_0 + \int_0^t\bd{T}(t-s)[\bd{F}(\hat{\bd{z}}(s),s)+\bd{Bu}(s) + \filt (s)(\bd{y}(s)-\hat{\bd{z}}(s))]\,ds.
\end{equation*}
\end{prop}
\begin{proof}
It suffices to note that since $\bd{\eta}(t)\in\mathcal{C}([0,t_f],\mathbb{R}^p)$, also $\bd{y}(t)=\bd{Cz}(t)+\bd{\eta}(t) \in \mathcal{C}([0,t_f],\mathbb{R}^p)$ and hence the nonlinear operator
\begin{equation}
\tilde{\bd{F}}(x,t)=\bd{F}(x,t)+\filt (t)[\bd{y}(t)-\bd{Cx}]
\end{equation}
 satisfies Assumption \ref{asp Frechet}. 
\end{proof}


The following generalization of semigroups is useful for time-varying systems.
\begin{defn}
Let $\Delta(t_f):=\{(t,s):0\leq s \leq t\leq t_f \}$. A mapping $\bd{U}(t,s):\Delta(t_f)\mapsto\mathcal{L}(\mathcal{H})$ is an evolution operator, if
\begin{enumerate}
\item[(i)] $\bd{U}(t,t)= \bd{I}$, $t\in\Delta(t_f)$, where $I$ denotes the identity operator,
\item[(ii)] $\bd{U}(t,r)\bd{U}(r,s)=\bd{U}(t,s)$, $0\leq s\leq r\leq t\leq t_f$,
\item[(iii)] $\bd{U}(\cdot,s)$ is strongly continuous on $[s,t_f]$ and $\bd{U}(t,\cdot)$ is strongly continuous on $[0,t]$.
\end{enumerate}
\end{defn}
For the proof of the following property, as well as more details on evolution operators, see for example,  \cite[section 5.3]{CZnew}.
\begin{thm}
\label{thm mild evol op}
Let $\bd{A}$ be the generator of a $C_0$-semigroup $\bd{T}(t)$ on $\mathcal{H}$ and let $\bd{D}(t)\in \mathcal{C}([0,t_f],\mathcal{L}(\mathcal{H}))$. Then there exists a unique evolution operator $\bd{U}(t,s)$ satisfying
\begin{equation}
\label{mild evol op}
\bd{U}(t,s)\bd{x}=\bd{T}(t-s)\bd{x}-\int_s^t\bd{T}(t-r)\bd{D}(r)\bd{U}(r,s)\bd{x}dr
\end{equation}
for all $\bd{x}\in\mathcal{H}$. We call $\bd{U}(t,s)$ the evolution operator generated by $\bd{A}+\bd{D}(t)$.
\end{thm}
For $\bd{f}(t)\in\mathcal{C}([s,t_f],\mathcal{H})$,
the mild solution to
\begin{align*}
\frac{\partial \bd{v}(t)}{\partial t}&=(\bd{A}+\bd{D}(t))\bd{v}(t)+\bd{f}(t)\\
\bd{v}(s)&=\bd{v}_0\in\mathcal{H}
\end{align*}
is  
$$
\bd{v}(t)=\bd{U}(t,s)\bd{v}_0 + \int_s^t\bd{U}(t,r)\bd{f}(r)dr.
$$

In the next section, a method of calculating the observer gain $\filt (t)$ based on the Extended Kalman filter is described.

\section{Definition and well-posedness of EKF}

The problem of observer design for linear systems has been well studied. The most widely known and used approach is the Kalman filter. 
Consider a time-varying linear system
\begin{align}
    \frac{\partial\z(t)}{\partial t} &= \tilde { \bd{A}} (t) \z (t)  + \bd{Bu}(t)  + \bd{\omega} (t) \label{eq-time-sys}.  \\
    \bd{y}(t) &= \bd{C} \bd{z}(t) + \bd{\eta} (t) . \nonumber
\end{align}
where $\tilde{\bd{A}}(t)$ generates an  evolution operator $\bd{U}_{\tilde A} (t,s)$ and $\bd{\omega}(t)$ and $\bd{\eta}(t)$ are process and output disturbance  respectively.
\begin{asp}
\label{asp operators}
 Let linear operators $\bd{P}_0\in \mathcal{L}(\ss)$, $\Wb(t)\in \C([0,t_f],\mathcal{L}(\ss))$  and $\bd{R}(t)\in \C([0,t_f],\mathcal{L}(\mathbb{R}^p))$ be self-adjoint. The operator $\bd{P}_0$ is positive definite; that is,  for every nonzero $\bd{v}\in\mathcal{H}$, $(\bd{v},\bd{P}_0\bd{v})>0$. Moreover, for all $t\in[0,t_f]$,  $\bd{W}(t)$ is nonnegative definite: for every nonzero $\bd{v}\in\mathcal{H}$, $(\bd{v},\bd{W}(t)\bd{v})\geq 0$. Finally, for all $t\in[0,t_f]$ the operator $\bd{R}(t)$ is uniformly coercive; meaning that there exists a $\delta_0>0$ such that for every $\bd{w}\in\mathbb{R}^p$ and $t\in [0,t_f]$, $(\bd{w},\bd{R}(t)\bd{w})\geq\delta_0\|\bd{w}\|^2$.
\end{asp}
Note that since $\bd{R}(t)$ is self-adjoint, coercive and bounded for each $t$, it follows that for all $t\in[0,t_f],$
  it has a self-adjoint bounded inverse $\bd{R}^{-1}(t)\in\mathcal{C}([0,t_f],\mathcal{H}).$
  
Linear {\em integral Riccati equations} are defined in the following theorem. For the proof we refer to \cite[Theorem 3.1 \& 3.3]{CRiccati-1976 } or \cite[Theorem 2.1]{curtain1976}.

\begin{thm}
\label{thm Riccati}
Let $\bd{U}(t,s)$ be an evolution operator on $\Delta(t_f)$. Under Assumption \ref{asp operators}, the following two integral Riccati equations
\begin{align}
\label{Riccati 1}
\bd{P}(t)\bd{w}&=\bd{U}_P(t,0)\bd{P}_0\bd{U}^\ast(t,0)\bd{w} +\int_0^t\bd{U}_P(t,s)\bd{W}(s)\bd{U}^\ast(t,s)\bd{w}ds,\quad\bd{w}\in\mathcal{H},\\
\label{LORiccati}
\bd{P}(t)\bd{w}&=\bd{U}_P(t,0)\bd{P}_0\bd{U}_P^*(t,0)\bd{w}+\int_0^t\bd{U}_P(t,s)(\bd{W}(s) +\bd{P}(s)\bd{C}^*\bd{R}^{-1}(s)\bd{C}\bd{P}(s))\bd{U}_P^*(t,s)\bd{w}ds,\quad\bd{w}\in\mathcal{H},
\end{align}
where the perturbed evolution operator
\begin{equation}
\label{Ladd2}
\bd{U}_P(t,s)\bd{w}=\bd{U}(t,s)\bd{w}-\int_{s}^{t}\bd{U}(t,r)\bd{P}(r)\bd{C}^*\bd{R}^{-1}(r)\bd{C}\bd{U}_P(r,s)\bd{w}dr,
\end{equation}
are equivalent and admit a unique, positive definite, self-adjoint solution $\bd{P}(t)\in\mathcal{C}([0,t_f],\mathcal{L}(\mathcal{H}))$.
\end{thm}
The solution $\bd{P}(t)$ of the Riccati equations \eqref{Riccati 1} and \eqref{LORiccati}, defines an observer gain,
 $$\filt (t) = \bd{P}(t)\bd{C}^\ast \bd{R}^{-1}(t).$$
 Letting $U$ be the evolution operator generated by  $\tilde {\bd{A}} ,$  observer dynamics for \eqref{eq-time-sys} are
 \begin{align*}
    \frac{\partial\hat{\bd{z}}(t)}{\partial t} &= \tilde{ \bd{A}}(t) \hat{\bd{z}}(t) + \bd{Bu}(t) + \filt (t) (\bd{y}(t) - C \bd z(t)  )  .
\end{align*}
In the case that  $\bd{\omega}(t)$ and $\bd{\eta}(t)$ are process and sensor noises with covariances $\bd{W}(t)$ and $\bd{R}(t)$ respectively and the covariance of the
initial condition $\bd{\hat{z}}(0)$ is $\bd{P}_0$ then the  observer gain $\filt (t)$ and corresponding estimate $\hat \z (t)$ are optimal in a sense that $\hat{\z} (t)$ minimizes the  error covariance \cite{curtain1976}.  The observer in this case is  Kalman filter. For details, see the book \cite{CurtainPritchard1978} and for recent work on time-varying systems, \cite{WuJacob2015}.

The linearization of \eqref{Sys1g} will be used to define a integral Riccati equation similar to \eqref{Riccati 1}, \eqref{LORiccati}. Their solution $\bd{P}$ define the observer gain. In the case of finite-dimensional systems, this approach is known as an  extended Kalman filter (EKF) and this terminology will be used here. 

First, the  linearization of the system is defined.
For this purpose, for $\hat{\bd{z}}(t) \in C([0,t_f], \mathcal{H} ) , $   at time $t$ the Fr\'echet-derivative of $\bd{F}(\cdot,t), $ denoted by  $\bd{D}\bd{F}(\cdot,t):\mathcal{H}\rightarrow\mathcal{L}(\mathcal{H})$, is
\begin{equation}
\label{DR}
\bd{D}\bd{F}(\hat{\bd{z}}(t),t)=\frac{\partial  \bd{F}(\bd{z}(t),t)}{\partial \bd{z}(t)}\mid_{\bd{z}(t)=\hat{\bd{z}}(t)}.
\end{equation} 
Linearizing the system \eqref{Sys1g} around $\hat{\bd{z}}(t)$ yields
\begin{equation}
\label{OBS-Lin}
\frac{\partial\bd{z}(t)}{\partial t}=\bd{A}\bd{z}(t)+\bd{F}(\hat{\bd{z}}(t),t) + \bd{D}\bd{F}(\hat{\bd{z}}(t),t)[\bd{z}(t)-\hat{\bd{z}}(t)]
+\bd{B}_d\bd{u}_d(t).
\end{equation} 
To obtain the EKF equations, the solution to  a Riccati equation for  the linear system \eqref{OBS-Lin}  is needed.   This Riccati equation will contain the Fr\'echet-derivative  $\bd{D}\bd{F}(\hat{\bd{z}}(t),t)$, a possibly nonlinear function of the observer state $\hat{\bd{z}}(t). $ It is for that reason, that the Riccati equation still depends on the observer state $\hat{\bd{z}}$, that recent results on the Riccati equation do not provide well-posedness for the coupled system we study here.

For a bounded linear operator $\bd{P}(t)\in C([0,t_f];\mathcal{L}(\mathcal{H}))$  and $ \zh_{0} \in \ss$, define  $\hat{\bd{z}}_P(t)$ as
\begin{equation}
\begin{split}
\zh_{P}(t)=\T(t)\zh_{0} &+ \int_0^t \T(t-s)\big(\Fb(\zh_P(s),s)+\bd{B}\bd{u}(s)\big)ds
\\
 &+ \int_0^t \T(t-s)\bd{P}(s)\bd{C}^\ast \bd{R}^{-1}(s)[\bd{y}(s)-\bd{C}\zh_P(s)]ds.
 \end{split}
\label{statespaceobserver_notime}
\end{equation}
If this equation has a unique solution for $\zh_P, $ it defines a mapping
$$
\zhP(t)  =\G_1(\bd{P}(t)).
$$
For $\hat{\z}_P(\cdot) \in C([0,t_f]; \ss),$   $\w \in \ss ,$  an evolution operator is defined as 
\begin{equation}
\U ( t,s)\bd{w}=\T(t-s)\bd{w}
+\int_{s}^{t}\T(t-r)\bd{D}\bd{F}(\hat \z_P (r),r) \U (r,s) \bd{w} dr .
\label{add1}
\end{equation}
This evolution operator is generated by the time-varying operator $\bd{A} +\bd{DF}(\hat \z_P  (t),t) . $
For each $\hat \z_P (\cdot) $   the associated time-varying system,  $\bd{w}\in \ss $ and the operators $\bd{P}_0, \bd R  , \bd W , $ imply an integral Riccati equation
\begin{align}
\label{add2}
\bd{U}_P(t,s)\bd{w}&=\bd{U}(t,s)\bd{w}  -\int_{s}^{t}\bd{U}(t,r)\bd{P}(r)\bd{C}^*\bd{R}^{-1}(r)\bd{C}\bd{U}_P(r,s)\bd{w}dr. 
\\
\label{ORiccati}
\bd{P}(t)\bd{w}&=\bd{U}_P(t,0)\bd{P}_0\bd{U}^\ast(t,0)\bd{w}   +\int_0^t\bd{U}_P(t,s)\bd{W}(s)\bd{U}^\ast(t,s)\bd{w}ds. 
\end{align}
  This defines a  second  mapping  
\begin{align*}
\bd{P}(t)&=\G_2(\zh_P (t)) . 
\end{align*}

It will be shown  that the mappings $\G_1 , $ $ \G_2$ are well-defined. Then it is shown  that the composite mapping $\G(\cdot)=\G_2(\G_1(\cdot))$   has a unique fixed point,
$\bd{P}(t)=\G(\bd{P}(t)).$
This will show  that the observer dynamics coupled with extended Riccati equations are well-posed in a sense that  equations \eqref{statespaceobserver_notime}, \eqref{add1}, \eqref{add2}, \eqref{ORiccati} have a unique solution $\zhP(t)\in C([0,t_f];\mathcal{H})$ and $\bd{P}(t)\in C([0,t_f];\mathcal{L}(\mathcal{H})) .$
The operator
\begin{equation}
\label{filter}
\filt (t)=\bd{P}(t)\bd{C}^*\bd{R}^{-1}(t)
\end{equation}
will then define an observer \eqref{OBD}.
The proof of well-posedness  is in \cite[Chapter 7]{Fabian} but is provided here for completeness.  

\begin{prop}
The mapping $G_1$ defined by 
$
\zhP(t)  =\G_1(\bd{P}(t))$
is   well-defined and  $\G_1:\C([0,t_f],\mathcal{L}(\ss)) \longrightarrow \C([0,t_f],\ss).$
\end{prop}

\begin{proof}
Due to $\bd{P}(t)\in\mathcal{C}([0,t_f],\mathcal{L}(\mathcal{H}))$ and  Assumption \ref{asp operators}, $\bd{L}(t)\in\mathcal{C}([0,t_f]; \mathcal{L}(\mathbb{R}^p,\mathcal{H}) ) . $ Hence, by Proposition \ref{WPness}  $\hat{\bd{z}}_P(t)\in\mathcal{C}([0,t_f],\mathcal{H})$.
\end{proof}

\begin{prop}
The mapping
\begin{equation*}
\begin{split}\G_2:\C([0,t_f],\ss)&\longrightarrow \C([0,t_f],\mathcal{L}(\ss))\\
\end{split}
\end{equation*}
is well-defined.
\end{prop}
\begin{proof}
It suffices to note that for $\bd{x}(t)\in \C([0,t_f],\ss)$, $\bd{D}(\bd{x}(t),t)\in\mathcal{C}([0,t_f],\mathcal{L}(\mathcal{H}))$ and therefore, by Theorem \eqref{thm mild evol op} the operator $\bd{A}+\bd{D}(\bd{x}(t),t)$ generates a evolution operator 
 \begin{equation}
\label{evolution_equation_1}
\U_{x} (t,s)\bd{w} = \T (t-s)\bd{w}+\int_s^t\T (t-r)\bd{D}\Fb(\bd{x}(r),r)\U_{x} (r,s)\bd{w} dr,\quad\bd{w}\in\mathcal{H}.
\end{equation}
By Theorem \ref{thm Riccati}, there is a unique, positive, self-adjoint $\bd{P}_x(t)\in\mathcal{C}([0,t_f],\mathcal{L}(\mathcal{H}))$ satisfying for all $\bd{w}\in \mathcal{H}$,
 \begin{align}
\label{perturbed_evol_operator}
\U_P(t,0) \bd{w}&=\U_{x}(t,0)\bd{w} -\int_0^t \U_{x}(t,r)\bd{P}_x(r)\bd{C}^\ast \bd{R}^{-1}(r)\bd{C}\U_P(r,0)\bd{w} dr\\
\bd{P}_x(t)\bd{w}  &= \U_P(t,0)\bd{P}_0\U_{x}^\ast (t,0)\bd{w} +\int_{0}^t \U_P(t,s)\bd{W}(s)\U_{x}^\ast (t,s)\bd{w} ds.
 \label{ire2}
\end{align}
Therefore, the mapping $\G_2(\cdot)$ is well-defined from its domain to the range. 
\end{proof}

The following result extends EKF based observer design to the class of semilinear infinite-dimensional systems considered in this paper.

\begin{thm}
\label{well-posed_thm_coupling}
Let Assumptions \ref{asp Frechet} and \ref{asp operators} hold. 
For any $\bd{u}(t)\in C([0,t_f],\mathcal{H}_1), $ $y(t)\in C([0,t_f],\mathbb{R}^p)$ and $\zh_0\in\ss$ 
 there exist $\zhP(t)\in\C([0,t_f],\ss)$ and $\bd{P}(t)\in\C([0,t_f],\mathcal{L}(\ss))$ such that $\zhP(t)$ solves \eqref{statespaceobserver_notime}
and $\bd{P}(t)$ satisfies the Riccati equation \eqref{ORiccati}, coupled to \eqref{add1} and  \eqref{add2}.
\end{thm}

\begin{proof}
For $\bd{x}=\zhP(t)$, define the mapping $\G=\G_2\circ\G_1$ by 
\begin{equation*}
\begin{split}
\G:\C([0,t_f],\mathcal{L}(\ss))&\longrightarrow\C([0,t_f],\mathcal{L}(\ss))\\
\bd{P}(t)&\mapsto \bd{P}_\zhP(t).
\end{split}
\end{equation*}
It will now be shown that 
$\G$ has a unique fixed point in $C([0,t_f]; \cL (\ss ) ) . $  This means that there is a unique pair $(\zh_P(t),\bd{P}(t))$ with $\zhP(t)$ satisfying the semilinear system \eqref{statespaceobserver_notime} and $\bd{P}(t)$ the Riccati equations \eqref{ORiccati}-\eqref{add2} or equivalently \eqref{perturbed_evol_operator} and \eqref{ire2} for $\bd{x}(t)=\zhP(t)$. This will imply  that  the semilinear system  \eqref{OBD} coupled with the Riccati equations \eqref{ORiccati}-\eqref{add1} has a unique solution.

The proof is divided into three steps:
\begin{enumerate}
\item Define the closed ball
$$\mathbb{P}_{t_f}(\delta_p)=\{ \bd{P}(t) \in C([0,t_f]; \mathcal{L} (\mathcal{H}))  , \;  \Vert\bd{P}(t)\Vert\leq\delta_p \} . $$
It will be shown that for  suitable  $\delta_p$  the mapping $\G$ maps the ball $\mathbb{P}_{t_f}(\delta_p)$ to itself, i.e. $\G:\mathbb{P}_{t_f}(\delta_p)\longrightarrow\mathbb{P}_{t_f}(\delta_p).$ 
\item   Show $\G^n$ is contractive on $\mathbb{P}_{t_f}(\delta_p)$ for large enough $n$.
\item  A fixed point argument concludes the proof.
\end{enumerate}

\textbf{Step 1:} In order to prove the first part of the proof, let $\delta_{DF}$ be the bound for $\bd{DF}(\cdot,\cdot) $ given by Assumption \ref{asp Frechet}.
 Choose $\bd{P}(t)\in \C([0,t_f],\mathcal{L}(\ss))$.
Using the Gr\"{o}nwall inequality, we can conclude from \eqref{evolution_equation_1} that for all $t\in [0,t_f]$,
\begin{equation}
\label{U_zh_bound}
\begin{split}
\max_{0\leq s\leq t}\|\Uzh(t,s)\|&\leq\delta_{T,\alpha} + \delta_{T,\alpha} \delta_{DF}\int_0^t \max_{0\leq s \leq r}\|\Uzh(r,s)\|dr\\
&\leq \delta_{T,\alpha}\exp(\delta_{T,\alpha}\delta_{DF} t_f)\\
&=: \delta_{U_{\zhP}},
\end{split}
\end{equation}
where $\delta_{T,\alpha}=\max_{t\in[0,t_f]}\|T_\alpha(t)\|$.
Now, let $\bd{P}(t)\in\mathbb{P}_{t_f}(\delta_p)$. Define $\delta_W=\max_{[0,t_f]}\|\bd{W}(t)\|$. By \cite[Lemma 2.2]{CRiccati-1976},  
\begin{equation*}
\begin{split}
\max_{[0,t_f]}\|\G(\bd{P}(t))\|&\leq(\|\bd{P}_0\|+t_f\delta_W)\delta_{\Uzh}^2\leq \delta_p
\end{split}
\end{equation*}
where the last inequality is true if $\delta_p$ is sufficiently large. 
For such $\delta_p$, we can conclude the first part of the proof, or 
\begin{equation*}
\G:\mathbb{P}_{t_f}(\delta_p)\longrightarrow\mathbb{P}_{t_f}(\delta_p).
\end{equation*}

\textbf{Step 2:} Now, we prove that $\G^n$ is  a contraction on $\mathbb{P}_{t_f}(\delta_p)$ for large enough $n\in\mathbb{N}$. For $i=1,2$, let $\bd{P}_i(t)\in\mathbb{P}_{t_f}(\delta_p)$ and define
\begin{description}
\item $\bullet$ $\hat{z}_P(t)=\G_1(\bd{P}_i(t))$ satisfying    \eqref{statespaceobserver_notime} with $\bd{P}(t)=\bd{P}_i(t)$
\item $\bullet$ $U_{\hat{z}_P,i}(t,s)$ being the evolution operator given by \eqref{evolution_equation_1} with $\bd{x}(t)=\hat{z}_P(t)$
\item $\bullet$ $U_{P,i}(t,s)$ being the perturbation of $\U_{\zhP,i}(t,s)$ by $-\G(\bd{P}_i(t))\bd{C}^\ast \bd{R}^{-1}(t)\bd{C}$ given by \eqref{perturbed_evol_operator}.
\end{description}
Moreover define 
\begin{align*}
\Delta \zhP(t)&=\zh_{P_1}(t)-\zh_{P_2}(t), &
\DUzh(t,s) &= \U_{\zhP,1}(t,s)-\U_{\zhP,2}(t,s),\\
\Delta\Pb(t)&=\Pb_1(t)-\Pb_2(t), &
\Delta \U_P(t,s)&=\U_{P,1}(t,s)-\U_{P,2}(t,s),\\
\Delta\bd{G}(t)&=\bd{G}(\bd{P}_1(t))-\bd{G}(\bd{P}_2(t)).
\end{align*}

\textbf{Step 2.1:} In this step, it is shown that the operators $\Delta \Ukc(t,s)$ and $\Delta U_{\zhP}(t,s)$ are bounded and the bounds are defined.

For $\zhP(t)$ satisfying \eqref{statespaceobserver_notime}, by Assumption \ref{asp Frechet} and the Gr\"{o}nwall inequality, it can be shown that there exists $\delta_\zhP>0$ such that for all $t\in [0,t_f]$,
\begin{equation}
\label{bound_zhP}
\|\zhP(t)\|\leq \delta_\zhP,\quad\text{for all }\bd{P}(t)\in\mathbb{P}_{t_f}(\delta_p).
\end{equation}
For the difference $\Delta\zhP$ we compute
\begin{equation}
\label{zpDel}
\begin{split}
\Delta\zhP(t)&=\int_0^t \bd{T}(t-s)\big(\bd{F}(\zh_{P_1}(s),s)-\bd{F}(\zh_{P_2}(s),s)\big)ds  + \int_0^t \bd{T}(t-s)\Delta \bd{P}(s)\bd{C}^\ast \bd{R}^{-1}(s) y(s)ds \\
&-\int_0^t \bd{T}(t-s)\big(\Delta \bd{P}(s)\bd{C}^\ast \bd{R}^{-1}(s)\bd{C}\zh_{P_1}(s)+\bd{P}_2(s)\bd{C}^\ast \bd{R}^{-1}(s)\bd{C}\Delta\zhP(s)\big)ds\\
\end{split}.
\end{equation}
From \eqref{zpDel}, Assumptions \ref{asp Frechet} and \ref{asp operators}, inequality \eqref{bound_zhP}, and boundedness of the operators $\bd{T}(t-s)$, $\bd{C}$, and $\bd{R}^{-1}(t)$,  it is concluded that there are constants $\delta_1,\delta_2>0$ and $c_1>0$ such that for all $t\in [0,t_f]$,
\begin{equation}
\label{bound_dzh}
\begin{split}
\|\Delta\zhP(t)\|&\leq  \delta_1\max_{0\leq s\leq t}\|\Delta \bd{P}(s)\| + 
 \delta_2\int_0^t \|\Delta\zhP(s)\|ds.\\
 &\leq c_1 \max_{0\leq s\leq t}\|\Delta \bd{P}(s)\|
\end{split}
\end{equation}
where the last inequality was obtained by applying the Gr\"{o}nwall inequality.

Similarly, from \eqref{evolution_equation_1}, it is derived that for all $\bd{w}\in\mathcal{H}$,  
\begin{equation}
\label{UzpDel}
\begin{aligned}
\Delta\U_{\zhP}(t,s)\bd{w} &= \int_s^t \T(t-r)\big(\DR(\zh_{P_1}(r),r)-\DR(\zh_{P_2}(r),r)\big)\U_{\zhP,1}(r,s)\bd{w}dr\\
&+\int_s^t\T(t-r)\DR(\zh_{P_2}(r),r)\Delta\U_\zhP(r,s)\bd{w}dr.
\end{aligned}
\end{equation}
Given Assumption \ref{asp Frechet}, the boundedness defined by \eqref{bound_zhP}, and the boundedness of the operators $\T(t-s)$ and $\U_{\zhP,1}(r,s)$ defined by \eqref{U_zh_bound}, \eqref{UzpDel} allow us to compute
\begin{equation*}
\begin{split}
\max_{0\leq s\leq t} \|\Delta\U_\zhP(t,s)\|&\leq \delta_{T,\alpha} \iota_{DF}\delta_{U_{\zhP}} \int_0^t \|\Delta\zhP(s)\| ds  \\
&+ \delta_{T,\alpha}\delta_{DF}\int_0^t \max_{0\leq s\leq r}\|\Delta\U_\zhP(r,s)\|dr.
\end{split}
\end{equation*}
which by Gr\"{o}nwall inequality yields for all $t\in [0,t_f]$,
\begin{equation}
\label{bound_dzh2}
\|\DUzh(t,s)\|\leq c_2\int_0^t \|\Delta\zhP(s)\| ds
\end{equation}
for some $c_2>0$. From \eqref{bound_dzh} and \eqref{bound_dzh2}, it is concluded that for all $t\in [0,t_f]$,
\begin{equation}
\label{bound_dUzh}
\|\DUzh(t,s)\|\leq c_2c_1\int_0^t \max_{0\leq \tau\leq s}\|\Delta \bd{P}(\tau)\|ds.
\end{equation}

Note that by the boundness of $\bd{U}_{\zhP,i}(t,s), i=1,2$ given by \eqref{U_zh_bound}, as well as Gr\"{o}nwall's lemma, it can be shown that the perturbed evolution operators $\bd{U}_{P,i}(t,s), i=1,2$ are, for all $t\in [0,t_f]$, also bounded by 
\begin{equation}
\label{U_p_bound}
\max_{0\leq s\leq t}\|\bd{U}_{P,i}(t,s)\|\leq \delta_{U_P},
\end{equation}
for some $\delta_{U_P}>0$.

Let $\bd{Q}(t)=\bd{C}^\ast \bd{R}^{-1}(t)\bd{C}$. Similarly, from \eqref{perturbed_evol_operator}, it can be derived that for all $\bd{w}\in\mathcal{H}$,
\begin{equation}
\label{UpiDel}
\begin{split}
\Delta\U_P(t,s)\bd{w} &= \Delta\U_\zhP(t,s)\bd{w} - \int_s^t\Delta\U_\zhP(t,r)\G(\bd{P}_1(r))\Qb(r)\U_{P,1}(r,s)\bd{w}dr\\
&-\int_s^t \U_{\zhP,2}(t,r)\Delta\G(r)\Qb(r)\U_{P,1}(r,s)\bd{w}dr -\int_s^t\U_{\zhP,2}(t,r)\G(\bd{P}_2(r))\Qb(r)\Delta\U_P(r,s)\bd{w}dr.
\end{split}
\end{equation}
Using the boundedness defined by \eqref{U_zh_bound} and \eqref{U_p_bound} as well as the boundedness of the operator $\bd{Q}(t)$, by Gr\"{o}nwall's inequality, we can show that \eqref{UpiDel} leads to 
\begin{equation*}
\max_{0\leq s\leq t}\|\Delta\U_P(t,s)\|\leq c_3 \max_{0\leq s\leq t}\|\Delta\U_{\zhP}(t,s)\|+ c_4 \int_0^t \|\Delta\G(r)\|dr
\end{equation*}
for some constants $c_3,c_4>0$; substituting \eqref{bound_dUzh} into this inequality yields for all $t\in [0,t_f]$,
\begin{equation}
\label{bound_dUp}
\|\Delta\U_P(t,s)\| \leq c_3 c_2c_1\int_0^t \max_{0\leq \tau\leq s}\|\Delta \bd{P}(\tau)\|ds + c_4 \int_0^t \|\Delta\G(r)\|dr.
\end{equation}

\textbf{Step 2.2: }It this step, we use the bounds obtained in the above to find a bound for $\G(\bd{P}_1(t))-\G(\bd{P}_2(t)).$
Note that the operators $\G(\bd{P}_i(t))$ for $i=1,2$ satisfy
\begin{equation}
\label{Peq}
\G(\bd{P}_i(t))\bd{w}  = \U_{P,i}(t,0)\bd{P}_0\U_{\zhP,i}^\ast (t,0)\bd{w} +\int_{0}^t \U_{P,i}(t,s)\bd{W}(s)\U_{\zhP,i}^\ast (t,s)\bd{w}ds
\end{equation}
for all $\bd{w}\in\ss$.
From \eqref{Peq}, it can be concluded that the difference $\G(P_1(t))-\G(P_2(t))$ satisfies, for all $\bd{w}\in\mathcal{H}$,
\begin{equation*}
\begin{split}
\big(\G(\bd{P}_1(t))&-\G(\bd{P}_2(t))\big)\bd{w} =\Delta \U_P(t,0)\bd{P}_0 \U_{\zhP,1}^\ast(t,0)\bd{w} + \U_{P,2}(t,0)\bd{P}_0 \Delta \U_{\zhP}^\ast(t,0)\bd{w}\\
&+\int_0^t \big(\Delta \U_P(t,s)\bd{W}(s)\U_{\zhP,1}^\ast(t,s) + \U_{P,2}(t,s)\bd{W}(s)\DUzh^\ast(t,s)\big)\bd{w}ds.
\end{split} 
\end{equation*}
which given \eqref{U_zh_bound} and \eqref{U_p_bound} as well as the boundedness of the operators $\bd{P}_0$ and $\bd{W}(t)$ leads to
\begin{equation}
\label{eqq}
\|\G(\bd{P}_1(t))-\G(\bd{P}_2(t)\|\leq c_5 \max_{0\leq s\leq t}\|\Delta \U_P(t,s)\| + c_5 \max_{0\leq s\leq t}\|\DUzh(t,s)\|.
\end{equation}
for all $t\in [0,t_f]$, for some $c_5>0$.

Substituting \eqref{bound_dUzh} and \eqref{bound_dUp} into \eqref{eqq} results in
\begin{equation*}
\|\G(\bd{P}_1(t))-\G(\bd{P}_2(t)\|\leq (c_5c_2c_1)(c_3+1)\int_0^t\max_{0\leq \tau\leq s}\|\Delta \bd{P}(\tau)\|ds + c_5c_4 \int_0^t \|\Delta\G(s)\|ds 
\end{equation*}
for all $t\in [0,t_f]$, and employing Gr\"{o}nwall's inequality yields for all $t\in [0,t_f]$,
\begin{equation}
\|\G(\bd{P}_1(t))-\G(\bd{P}_2(t))\|\leq c\int_0^t\max_{0\leq \tau\leq s}\|\bd{P}_1(\tau)-\bd{P}_2(\tau)\|ds.
\label{formel_G}
\end{equation}
for some $c>0$.

Now, we show that $\G^n(\cdot)$ is a contraction mapping for large enough $n>0$. For this purpose, we use an induction argument to show that
\begin{equation}
\|\G^n(\bd{P}_1(t))-\G^n(\bd{P}_2(t))\|\leq \frac{(ct)^n}{n!}\max_{0\leq\tau\leq t}\|\bd{P}_1(\tau)-\bd{P}_2(\tau)\|.
\label{induction_formel}
\end{equation}
First, for $n=1$ the argument holds by \eqref{formel_G}. Now let $n\geq 2$ and assume \eqref{induction_formel} holds for $k\leq n-1$. By \eqref{formel_G}
\begin{equation*}
\begin{split}
\|\G^n(\bd{P}_1(t))-\G^n(\bd{P}_2(t))\|&\leq c\int_0^t \max_{0\leq \tau\leq s}\|\G^{n-1}(\bd{P}_1(\tau))-\G^{n-1}(\bd{P}_2(\tau))\|ds \\
&\leq \frac{c^n}{(n-1)!}\int_0^t s^{n-1} \max_{0\leq\tau\leq s}\|\bd{P}_1(\tau)-\bd{P}_2(\tau)\|ds\\
&\leq \frac{c^n}{(n-1)!} \max_{0\leq \tau\leq t}\|\bd{P}_1(\tau)-\bd{P}_2(\tau)\|\int_0^t s^{n-1}ds\\
&\leq \frac{(ct)^n}{n!}\max_{0\leq \tau\leq t}\|\bd{P}_1(\tau)-\bd{P}_2(\tau)\|
\end{split}
\end{equation*}
which proves \eqref{induction_formel}. By taking the maximum on $[0,t_f]$ it follows that
\begin{equation*}
\max_{[0,t_f]}\|\G^n(\bd{P}_1(t))-\G^n(\bd{P}_2(t))\|\leq\frac{(ct_f)^n}{n!}\max_{[0,t_f]}\|\bd{P}_1(t)-\bd{P}_2(t)\|.
\end{equation*}
Therefore, for $n$ large enough $\frac{(ct_f)^n}{n!}<1$ and $\G^n(\cdot)$ is a contraction on $\mathbb{P}_{t_f}(\delta_p)$. 
Thus, by the Contraction Mapping Theorem,  see for example,  \cite[Lemma 5.4-3]{kreyszig}, there is a unique fixed point on $\mathbb{P}_{t_f}(\delta_p), $ completing the proof.
\end{proof}

This implies the existence of a unique mild solution of the observer dynamics \eqref{OBD} with observer gain $\filt (t)$  defined by \eqref{filter} where the linear operator $\bd{P}(t)$ is the solution  of the Riccati equation  \eqref{ORiccati}-\eqref{add1}. 
It may be convenient to adjust the rate at which the estimation error converges to zero, which will be discussed in more detail in the next section. To this end, in the Riccati coupled equations \eqref{add1},\eqref{add2}, and \eqref{ORiccati}, the operator $\bd{A}$ is replaced by $\bd{A}+\alpha\bd{I}$ where $\alpha>0$ and $\bd{I}$ is the identity operator. In other words, let $\Ta(t)$ be the $C_0$-semigroup generated by $\bd{A}+\alpha\bd{I}$ and for
$\hat{\z}_P(t) \in C([0,t_f]; \ss)$ and $\w \in \ss $ defines  a evolution operator 
\begin{equation}
\U_{\alpha} ( t,s)\bd{w}=\Ta(t-s)\bd{w}
+\int_{s}^{t}\Ta(t-r)\bd{D}\bd{F}(\hat \z_P (r),r) \U_{\alpha} (r,s) \bd{w} dr .
\label{add1a}
\end{equation}
The integral Riccati equation takes the form
\begin{align}
\label{add2a}
\bd{U}_{P,\alpha}(t,s)\bd{w}&=\bd{U}_{\alpha}(t,s)\bd{w}  -\int_{s}^{t}\bd{U}(t,r)\bd{P}(r)\bd{C}^*\bd{R}^{-1}(r)\bd{C}\bd{U}_{P,\alpha}(r,s)\bd{w}dr. 
\\
\label{ORiccatia}
\bd{P}(t)\bd{w}&=\bd{U}_{P,\alpha}(t,0)\bd{P}_0\bd{U}_{P,\alpha}^\ast(t,0)\bd{w}   +\int_0^t\bd{U}_{P,\alpha}(t,s)\left(\bd{W}(s)+\bd{P}(s)\bd{C}^*\bd{R}^{-1}(s)\bd{C}\bd{P}(s)\right)\bd{U}_{P,\alpha}^\ast(t,s)\bd{w}ds. 
\end{align}
Note that \eqref{add1a}-\eqref{ORiccatia} satisfy the condition of the conditions of Theorem \ref{well-posed_thm_coupling} and thus are well-posed on every bounded time interval $[0,t_f]$.

\section{Dynamics of the estimator}

Now the estimator defined by the solution to  the coupled integral Riccati equations discussed in the previous sections is analyzed as a dynamical system.
The following definition is standard.
\begin{defn}
An evolution operator $\bd{Y}(t,s)$ is {\em exponentially stable} if there exists $M \geq 0 , $ $\alpha >0$ such that for all $t\geq s ,$
$$\| Y(t,s) \| \leq M e^{-\alpha (t-s) } .$$
\end{defn}

Let $\bd{A}_c(t):\mathcal{D}(\bd{A}(t))\mathcal{H}\to\mathcal{H}$ be a linear operator that generates an evolution operator $\bd Y(t,s)$  and consider bounded  linear operators $\bd{B}_c(t) \in  \mathcal{C} \left( [0,\infty); \mathcal L (\mathcal{X}_B, \mathcal H)  \right)$  and $\bd{C}_c(t) \in  \mathcal{C} \left([0,\infty ); \mathcal L ( \mathcal H , \mathcal{X}_C) \right) $  where $\mathcal{X}_B$ and $\mathcal{X}_C$ are Hilbert spaces.    

\begin{defn}
\begin{enumerate}
\item
\label{Ad-det}
For a linear bounded and continuous in time operator $L (t)$ define 
 the evolution operator $\bd{Y}_{L } (t,s)$ satisfying
\begin{align}
\bd{Y}_{L }(t,s)\bd{w}&=\bd{Y} (t,s)\bd{w} \label{eq-YK0} \\
&+\int_s^t\bd{Y} (t,r) L (r)\bd{C}_c(r)\bd{Y}_{L}(r,s)\bd{w}dr, \quad \bd{w}\in\mathcal{H} .  \nonumber
\end{align}
The pair $(\bd{A}_c (t),\bd{C}_c(t))$ is {\em uniformly detectable} if there exists a linear uniformly  bounded and continuous in time operator $
L (t)$ such that  the evolution operator $\bd{Y}_{L}(t,s)$ is exponentially stable . 
\item
\label{Ad-stab}
For any linear uniformly bounded and continuous in time operator $
K(t) $ 
 define the evolution operator $\bd{Y}_{ K}(t,s)$ satisfying
\begin{align}
\bd{Y}_{K}(t,s)\bd{w}&=\bd{Y} (t,s)\bd{w}   \label{eq-YK1} \\
&+\int_s^t\bd{Y} (t,r)\bd{B}_c(r)\bd{K}(r)\bd{Y}_{K}(r,s)\bd{w}dr,  \quad \bd{w}\in\mathcal{H}. \nonumber
\end{align}
The pair $(\bd{A}_c(t),\bd{B}_c(t))$ is  {\em uniformly stabilizable} if there exists  a linear uniformly bounded and continuous in time operator $
\bd{K}(t)$ such that the evolution operator $\bd{Y}_{K}(t,s)$ is exponentially stable . 
\end{enumerate}
\end{defn}

Duality with the  corresponding linear quadratic control problem will be used to  establish the main result. First, we introduce an optimal control problem. Let 
$\bd{P}_0,$  $\bd{W}_{c}(t)$ and $\bd{R}_{c}(t)$  be operators satisfying  Assumption
 \ref{asp operators}. 
 For $t_f > 0 $ consider the cost function
\begin{equation}
\label{Cost}
\begin{aligned}
J(\bd{u}(t);\bd{z}_{adj,0},0)&=(\bd{z}_{adj}(t_f),\bd{P}_0\bd{z}_{adj}(t_f))_{\mathcal{H}}\\
&+\int_{0}^{t_f}\big((\bd{z}_{adj}(r),\bd{W}_{c}(r)\bd{z}_{adj}(r))_{\mathcal{H}}
+(\bd{u}(r),\bd{R}_{c}(r)\bd{u}(r))_{\mathcal{X}_B}\big)dr.
\end{aligned}
\end{equation}
subject to the dynamics
\begin{equation}
\label{Osysd}
\frac{\partial z_{adj}(t)}{\partial t} = \bd{A}_c (t) \bd{z}_{adj} (t) + \bd{B}_c(t) u(t) ,  \quad \bd{z}_{adj} (0) =\bd{z}_{adj,0}
\end{equation}
with mild solution
\begin{equation}
\label{Osys}
\begin{aligned}
\bd{z}_{adj}(t)&=\bd{Y} (t,0)\bd{z}_{adj,0}+\int_{0}^t\bd{Y} (t,r)\bd{B}_c(r)\bd{u}(r)dr.
\end{aligned}
\end{equation} 

\begin{thm}
\label{thm-lq}
 \cite[Thms. 2.1- 2.3]{CRiccati-1976 }
For any $z_0 \in \ss , $ the  integral equations
\begin{align}
\label{ORiccatiD}
\bd{P}_{t_f}(0)\bd{z_0}&=\Ykc^*(t_f,0)\bd{P}_0\Ykc(t_f,0)\bd{z_0}\\
& \quad +\int_{0}^{t_f}\Ykc^*(r,0)(\bd{W}_{c}(r)
+\bd{P}_{t_f}(r)\bd{B}_c(r)\bd{R}_{c}^{-1}(r)\bd{B}_c^*(r)\bd{P}_{t_f}(r))\Ykc (r,0)\bd{z_0}dr,  \nonumber\\
\Ykc(t,s)\bd{z_0}& =\bd{Y}(t,s)\bd{z_0}
-\int_s^t\bd{Y}(t,r)\bd{B}_c(r)\bd{R}_{c}^{-1}(r)\bd{B}_c^*(r)\bd{P}_{t_f}(r)\Ykc (r,s) z_0 dr,
\label{Yp}
\end{align}
have a unique solution $\bd{P}_{t_f}(\cdot ) \in \mathbb B^\infty \left( [0 , t_f]; \mathcal{L}(\mathcal{H})\right)$.
Furthermore,  the control
$$\bd{u}_{opt}(t)=-\bd{R}_{c}^{-1}(t)\bd{B}_c^*(t)\bd{P}_{t_f}(t)\bd{z}_{adj}(t),$$
minimizes the cost \eqref{Cost} and the minimum cost is 
\begin{equation}
\label{CostP}
\begin{aligned}
J(\bd{u}_{opt}(t);\bd{z}_{adj,0},0)&=(\bd{P}_{t_f}(0)\bd{z}_{adj,0},\bd{z}_{adj,0})_{\mathcal{H}}\end{aligned}
\end{equation}
\end{thm}

The optimal control problem  is now considered on  the infinite time interval. Assume that $A_c (t), B_c (t) $ and $C_c (t)$ are all linear operators defined for $t \in \mathbb R $ and also that $B_c (t) , C_c (t)$ are bounded uniformly in time. 
 Let 
$\bd{P}_0\in\mathcal{L}(\mathcal{H})$,   $\bd{W}_{c}(t)\in\mathcal{C}(\mathbb{R},\mathcal{L}(\ss))$ and $\bd{R}_{c}(t)\in\mathcal{C}(\mathbb{R},\mathcal{L}(\mathcal{Y}))$  be operators satisfying  Assumption
 \ref{asp operators}. Both $\bd{W}_{c}(t)$ and $\bd{R}_{c}(t)$ are in addition  assumed to be uniformly bounded uniformly in time and $\bd R_c (t)$ is coercive uniformly in time. 
The cost function becomes
\begin{equation}
\label{J-inf}
J_{\infty}(\bd{u}(t);\bd{z}_{adj,0},0)=\lim_{t_f\rightarrow\infty}\int_{0}^{t_f}\big((\bd{z}_{adj}(r),\bd{W}_{c}(r)\bd{z}_{adj}(r))_{\mathcal{H}}+(\bd{u}(r),\bd{R}_{c}(r)\bd{u}(r))_{\mathcal{X}_B}\big)dr
\end{equation}
subject to the same dynamics \eqref{Osys}.

\begin{thm}
\label{J-infty}
\cite[Thms. 4.2-4.4]{Gibson}
Assume that  for every  initial condition $\bd{z}_{adj,0}\in\mathcal{H}$ in \eqref{Osys}, there exists a stongly measurable control input $\bd{u}(t)$  such that the cost function  \eqref{J-inf} is finite. Let  $\bd{P}_{t_f}(\cdot)$ be the solution to the Riccati equation  \eqref{ORiccatiD} and \eqref{Yp} with $P_0=0$. Then there exists a unique nonnegative and self-adjoint  operator $\bd{P}_{\infty}(0)\in\mathcal{L}(\mathcal{H})$ such that, for any $z_0 \in \ss , $
\begin{align}
\min_{\bd{u}}J_{\infty}(\bd{u}(t);\bd{z}_{adj,0},0)&=(\bd{P}_{\infty}(0)\bd{z}_{adj,0},\bd{z}_{adj,0})_{\mathcal{H}} \,  , \label{eq-Gibson1}\\
\lim_{t_f \to \infty} \bd{P}_{t_f}(0)\bd{z}_{adj,0}  &= \bd{P}_{\infty}(0)\bd{z}_{adj,0} .  \label{eq-Gibson-2}
\end{align}
 for any $z_0 \in \ss , $ $0 \leq s \leq t . $
Furthermore,  for any $0 \leq s \leq t , $ $P_\infty(s) $ solves
\begin{align}
\label{eq-Pinf}
\bd{P}_{\infty}(s)\bd{z_0}&=Y_{-B_c K_\infty}^*(t,s)\bd{P_\infty}(t)Y_{-B_c K_\infty} (t,s)\bd{z_0}\\
& \quad +\int_{s}^{t}Y_{-B_c K_\infty}^*(r,s)(\bd{W}_{c}(r)
+\bd{P}_{\infty}(r)\bd{B}_c(r)\bd{R}_{c}^{-1}(r)\bd{B}_c^*(r)\bd{P}_{\infty}(r))Y_{-B_c K_\infty} (r,s)\bd{z_0}dr,  \nonumber\\
Y_{-B_c K_\infty}(t,s)\bd{z_0}& =\bd{Y}(t,s)\bd{z_0}
-\int_s^t\bd{Y}(t,r)\bd{B}_c(r)\bd{R}_{c}^{-1}(r)\bd{B}_c^*(r)\bd{P}_{\infty}(r)Y_{-B_c K_\infty} (r,s) z_0 dr,
\label{Ybk}
\end{align}
and $Y_{-B_c K_\infty}(t,s) $ is an evolution operator. 
\end{thm}

The linear operator $A_c(t)$ can be reformulated as 
$$A_c(t)=A_c(t)+L(t)W^{1/2}(t)-L(t)W^{1/2}(t)$$ where $L(t)$ is such that the evolution operator generated by $A_c(t)+L(t)W^{1/2}(t)$ is exponentially stable in the sense of definition \eqref{eq-YK0}. Let $Y_{L}(t,s)$ be the evolution operator generated by $A_c (t)+L(t)W^{1/2}(t)$. 
Since the dynamics \eqref{Osysd} can be represented as 
$$
\frac{\partial z_{adj}(t)}{\partial t}  = \big(A_c(t)+L(t)W^{1/2}(t)-L(t)W^{1/2}(t)\big)z_{adj}(t) + \bd{B}_c(t) u(t) ,  \quad \bd{z}_{adj} (s) =\bd{z}_{adj,0} ,
$$
$z_{adj} $ can be written 
\begin{equation}
\label{Osys-r}
z_{adj}(t)=Y_L(t,s)z_{adj,0}+\int_s^tY_L(t,r)\left(-L(r)W^{1/2}(r)z_{adj}(r)+B_c(r)u(r)\right)dr.
\end{equation}
This formulation will be useful in proofs of the following theorems.

\begin{thm}
\label{Y-Convergence}
Consider the  infinite-time optimal control problem \eqref{J-inf} and finite-time optimal control problem \eqref{Cost}  with dynamics \eqref{Osys}. 
 Assume that $(\bd{A}_c (t),\bd{W}^{1/2}_{c}(t))$ is uniformly detectable. Let $Y_{-B_cK_{\infty}}(t,s)$ be the evolution operator generated by the perturbation of $A_c (t) $ by  $- B_c (t) R_c^{-1}(t)B_c^*(t)P_{\infty}(t)$, \eqref{Ybk}, and similarly $Y_{-B_cK_{t_f}}(t,s)$ is the evolution operator generated by the perturbation of $A_c (t) $  by $- B_c (t) R_c^{-1}(t)B_c^*(t)P_{t_f}(t)$,\eqref{Yp}. For each $z_0\in\mathcal{H}$,
\begin{equation}
\label{Y-conv}
\lim_{t_f \to \infty } Y_{-B_cK_{t_f}}(t,s)z_0= Y_{-B_cK_{\infty}}(t,s)z_0 
\end{equation}  
 uniformly in  time for $0\leq s<t\leq\infty .$
\end{thm}

\begin{proof}
Define 
\begin{align*}
u_{\infty}(t) &=-R_c^{-1}(t)B_c(t)^* P_{\infty}(t)z_{adj,\infty}(t), \\ u_{t_f}(t) &=-R_c^{-1}(t)B_c(t)^*P_{t_f}(t)z_{adj,t_f}(t) 
\end{align*}
where $z_{adj,\infty}(t)$ and $z_{adj,t_f}(t)$ are the solutions to \eqref{Osys-r} with  $u(t)=u_{\infty}(t)$ and $u(t)=u_{t_f}(t)$ respectively. Since $R_c(t)$ is  uniformly lower and upper bounded, \cite[pg 551]{Gibson} implies 
\begin{eqnarray}
J(u_{t_f},z_{adj,0},0)&\rightarrow & J_{\infty}(u_{\infty}(t),z_{adj,0},0), \nonumber \\
\label{Osys-conv1}
u_{t_f}&\rightarrow & u_{\infty}\qquad \text{in}\quad \mathbb{L}^2([0,\infty];\mathcal{X}_B), \\
\label{Osys-conv2}
W^{1/2}(t)z_{adj,t_f}(t) & \rightarrow & W^{1/2}(t)z_{adj,\infty}(t)\quad \text{in}\quad \mathbb{L}^2([0,\infty];\mathcal{H}).
\end{eqnarray}

Using \eqref{Osys-r}, 
\begin{equation}
\label{Osys-er}
\begin{aligned}
\left(z_{adj,\infty}(t)-z_{adj,t_f}(t)\right)&=\int_s^tY_L(t,r)\bigg(-L(r)W^{1/2}(r)\left(z_{adj,\infty}(r)-z_{adj,t_f}(r)\right)\\
&+B_c(r)\left(u_{\infty}(r)-u_{t_f}(r)\right)\bigg)dr
\end{aligned}
\end{equation}
Define $\delta z_{adj}(t)=z_{adj,\infty}(t)-z_{adj,t_f}(t)$ and $\delta u(t)=u_{\infty}(t)-u_{t_f}(t)$.  
From \eqref{Osys-er}, 
\begin{equation}
\label{Osys-er-inq}
\begin{aligned}
\sup_{t\in[s,\infty)}\|\delta z_{adj}(t)\|_{\mathcal{H}}&\leq\int_s^\infty\|Y_L(t,r)\|\big(\sup_{r'}\|L(r')\|\|W^{1/2}(r)\delta z_{adj}(r)\|_{\mathcal{H}}\\
&+\sup_{r'}\|B_c(r')\|\|\delta u(r)\|_{\mathcal{U}}\big)dr\\
&\leq\|Y_L(t,s)\|_{\mathbb{L}^2(s,\infty)}\sup_{r'}\|L(r')\|\|W^{1/2}(t)\delta z_{adj}(t)\|_{\mathbb{L}^2([s,\infty];\mathcal{H})}\\
&+\|Y_L(t,s)\|_{\mathbb{L}^2(s,\infty)}\sup_{r'}\|B_c(r')\|\|\delta u(t)\|_{\mathbb{L}^2([s,\infty];\mathcal{X}_B)}.
\end{aligned}
\end{equation}
Since $Y_{L}(t,s)$ is  exponentially stable,    the convergence in   \eqref{Osys-conv2} into \eqref{Osys-er-inq} implies that $\sup_{t\in[s,\infty)}\|\delta z_{adj}(t)\|_{\mathcal{H}} \to 0 . $ This is equivalent to  \eqref{Y-conv} .
\end{proof}

\begin{thm}
\label{thm-inf-stab}
Consider the  infinite-time optimal control problem \eqref{J-inf}  with dynamics \eqref{Osys}.
 Assume that  $(\bd{A}_c (t),\bd{B}_c (t) )$ is uniformly  stabilizable and $(\bd{A}_c (t),\bd{W}^{1/2}_{c}(t))$ is uniformly detectable.
With the optimal  feedback  $$K_{\infty}(t)=\bd{R}_{c}^{-1}(t)\bd{B}_c^*(t) P_\infty (t), $$ the evolution operator $Y_{-B_c K_\infty}(t,s)$ defined in \eqref{Ybk}  is exponentially stable.  
\end{thm}

\begin{proof}
The optimal control is $$\bd{u}_{opt}(t)=-K_{\infty}(t)\bd{z}_{adj}(t). $$ 
 The dynamics of the controlled system are $\bd{z}_{adj}(t)=Y_{-B_c K_\infty}(t,s)\bd{z}_{adj,0}.$ 
 Define $\bar{\bd{R}}_{c}(t)=\bd{B}_c(t)\bd{R}_{c}^{-1}(t)\bd{B}_c^*(t)$. By optimality of $\bd{u}_{opt}(t)$, Theorem \ref{J-infty}, and definition of the cost function \eqref{J-inf},
 for any $t_f \geq 0 ,$ $ \bd{z}_{adj,0} \in \ss , $
 \begin{equation}
 \label{WR-bdd}
 \begin{aligned}
\int_{0}^{t_f}\big\Vert\bd{W}_{c}^{1/2}(r)\bd{z}_{adj}(r)\big\Vert^2_{\mathcal{H}}ds
+\int_{0}^{t_f}\big\Vert\bar{\bd{R}}_{c}^{1/2}(r)\bd{P}_{\infty}(r)\bd{z}_{adj}(r)\big\Vert^2_{\mathbb{H}}dr \leq (\bd{P}_{\infty}(0)\bd{z}_{adj,0},\bd{z}_{adj,0})_{\mathcal{H}}\leq \|P_{\infty}(0)\|\|z_{adj,0}\|_{\mathcal{H}}^2 .
\end{aligned}
\end{equation}
This implies that 
\begin{equation}
\label{Cost-stb}
\begin{aligned}
\bd{W}_{c}^{1/2}(t)\bd{z}_{adj}(t)&\in\mathbb{L}^2([0,\infty);\mathcal{H}), \\
\bar{\bd{R}}_{c}^{1/2}(t)\bd{P}_{\infty}(t)\bd{z}_{adj}(t)&\in\mathbb{L}^2([0,\infty);\mathcal{H}) . 
\end{aligned}
\end{equation}

Since $(\bd{A}_c(t),\bd{W}^{1/2}_{c}(t))$ is uniformly detectable, there is a linear  uniformly bounded operator $\bd{L}(t)$ such that the evolution operator $\bd{Y}_{L }(t,s)$, as defined in \eqref{eq-YK0} but with $C_c(t) = \bd{W}^{1/2}(t) ,$  is   exponentially stable. 
The operator $\bd{A}_c(t)-\bd{B}_c(t)K_{\infty}(t)$ can be written
\begin{align*}
\bd{A}_c (t)-\bd{B}_c(t)K_{\infty}(t)
&=\bd{A}_c (t)-\bar{\bd{R}}_{c}(t)\bd{P}_{\infty}(t)\\
&=\bd{A}_c (t)-\bd{L}_c(t)\bd{W}_{c}^{1/2}(t)\\
&\quad + \bd{L}_c(t)\bd{W}_{c}^{1/2}(t)-\bar{\bd{R}}_{c}(t)\bd{P}_{\infty}(t). 
\end{align*}  
Therefore, the  evolution operator $Y_{-B_c K_\infty}(t,s)$ generated  by $\bd {A}_c (t) -\bd{B}_c(t)\bd{K}_{\infty}(t)$ can be written 
\begin{equation}
\label{Eval-Y_Pi}
\begin{aligned}
Y_{-B_c K_\infty}(t,s)&=\bd{Y}_{L }(t,s)\\
&+\int_{s}^t\bd{Y}_{L }(t,r)\big(\bd{L}_c(r)\bd{W}_c^{1/2}(r)
-\bar{\bd{R}}_c(r)\bd{P}_{\infty}(r)\big)Y_{-B_c K_\infty}(r,s)dr.
\end{aligned}
\end{equation}
 
For  any initial condition $ \bd{v}\in\mathcal{H}$, \begin{equation}
\label{Eval-Y-Pi-ine}
\begin{aligned}
\Vert Y_{-B_c K_\infty}(t,s)\bd{v}\Vert_{\mathcal{H}}&\leq\Vert\bd{Y}_{L}(t,s)\bd{v}\Vert_{\mathcal{H}}\\
&+\int_{s}^t\Vert\bd{Y}_{L}(t,r)\Vert\Vert\bd{L}(r)\Vert\Vert\bd{W}_c^{1/2}(r)Y_{-B_c K_\infty}(r,s)\bd{v}\Vert_{\mathcal{H}}dr\\
&+\int_{s}^t\Vert\bd{Y}_{L }(t,r)\Vert\Vert\bar{\bd{R}}_c(r)^{1/2}\Vert  \; \Vert\bar{\bd{R}}_c(r)^{1/2}\bd{P}_{\infty}(r)\big)Y_{-B_c K_\infty}(r,s)\bd{v}\Vert_{\mathcal{H}} dr.
\end{aligned}
\end{equation}
Let $\delta_{Y,0}  , \alpha_Y >0$ be such that  $\Vert\bd{Y}_{L }(t,r)\Vert\leq\delta_{Y,0}\exp\left(-\alpha_Y(t-r)\right) .$  
Since $\bd{L}_c (t)$ and $\bar{\bd{R}}(t)$  are uniformly bounded over $[0,t_f]$, for some $\delta_{Y,0},\delta_{Y,1},\delta_{Y,2}>0$,
\begin{equation}
\label{Eval-Y-Pi-ine-d}
\begin{aligned}
\Vert Y_{-B_c K_\infty}(t,s)\bd{v}\Vert_{\mathcal{H}}&\leq\delta_{Y,0}\exp(-\alpha_Y(t-s))\|\bd{v}\|_{\mathcal{H}}\\
&+\delta_{Y,1}\int_{s}^t\exp\left(-\alpha_Y(t-r)\right)\Vert\bd{W}_c^{1/2}(r)Y_{-B_c K_\infty}(r,s)\bd{v}\Vert_{\mathcal{H}}dr\\
&+\delta_{Y,2}\int_{s}^t\exp\left(-\alpha_Y(t-r)\right)
\Vert\bar{\bd{R}}_c(r)^{1/2}\bd{P}_{\infty}(r)\big)Y_{-B_c K_\infty}(r,s)\bd{v}\Vert_{\mathcal{H}} dr.
\end{aligned}
\end{equation}
Since $\bd{W}_c^{1/2}(r)Y_{-B_c K_\infty}(r,s)\bd{v}\in\mathbb{L}^2([0,\infty);\mathcal{H})$ and $\bar{\bd{R}}_c^{1/2}(r)\bd{P}_{\infty}(r)Y_{-B_c K_\infty}(r,s)\bd{v}\in\mathbb{L}^2([0,\infty);\mathcal{H})$ \eqref{Cost-stb}, 
the convolution product  
\begin{equation}
\label{L2}
\Vert Y_{-B_c K_\infty}(t,s)\bd{v}\Vert_{\mathbb{L}^2([0,\infty);\mathcal{H})}\leq\delta_{Y,3},
\end{equation}
where $\delta_{\delta,3}>0$ is independent of $t_f>0.$
(See for example \cite[Lemma A.6.6]{Curtainb}.)
Furthermore,
 $Y_{-B_c K_\infty}$ is a continuous evolution operator  by \cite[Thm 1.1]{CRiccati-1976} and has an exponential growth bound. Therefore, by Datko’s theorem for evolution operators, \cite[Thm. 1]{Datko},  $Y_{-B_cK_{\infty}}(t,s)$ is exponentially stable.
\end{proof}

\begin{thm}
\label{thm-stab-control}
 Assume that  $(\bd{A}_c (t),\bd{B}_c (t) )$ is uniformly  stabilizable and $(\bd{A}_c (t),\bd{W}^{1/2}_{c}(t))$ is uniformly detectable.
Then the solution $P_{\infty}(s)$ to the Riccati equation \eqref{eq-Pinf} is unique and yields the  optimal cost.
\end{thm}
\begin{proof}
Let  $u(t)$ be any admissible control input strongly measurable control input such that \eqref{J-inf} is finite. The cost function \eqref{J-inf} is bounded above and thus for some $\delta_{u}>0$,
\begin{equation}
\label{ADM}
\begin{aligned}
\lim_{t_f\rightarrow\infty}\int_s^{t_f}(u(r),R_c(r)u(r))_{\mathcal{X}_B}dr\leq &\delta_u,\\
\lim_{t_f\rightarrow\infty}\int_s^{t_f}(z_{adj}(r),W_c(r)z_{adj}(r))_{\mathcal{X}_B}dr\leq &\delta_u.
\end{aligned}
\end{equation}
 Since $R_c(t)$ is uniformly coercive, 
\begin{equation}
\label{u_b}
u(t)\in\mathbb{L}^2([s,\infty);\mathcal{X}_B).
\end{equation}

Using \eqref{Osys-r} and letting $M_L , $ $\alpha_L$ be such that  $\|Y_L(t,s)\|\leq M_L\exp\left(-\alpha_L(t-s)\right),$
\begin{equation}
\label{Osys-er-inq-p}
\begin{aligned}
\|z_{adj}(t)\|_{\mathcal{H}}&\leq
\|Y_L(t,s)\|\|z_{adj,0}\|_{\mathcal{H}}\\
&+\int_s^t\|Y_L(t,r)\|\|L(r)\|\|W^{1/2}(r) z_{adj}(r)\|_{\mathcal{H}}dr\\
&+\int_s^t\|Y_L(t,r)\|\|B_c(r)\|\| u(r)\|_{\mathcal{U}}dr\\
&\leq M_L \|z_{adj,0}\|_{\mathcal{H}}\exp\left(-\alpha_L(t-s)\right)\\
&+M_L\sup_{r}\|L(r)\|\int_s^t\exp\left(-\alpha_L(t-r)\right)\|W^{1/2}(r) z_{adj}(r)\|_{\mathcal{H}}dr\\
&+M_L\sup_{r}\|B_c(r)\|\int_s^t\exp\left(-\alpha_L(t-r)\right)\| u(r)\|_{\mathcal{U}}dr.
\end{aligned}
\end{equation}
The convolution product  \eqref{Osys-er-inq-p}  bounding $\| z_{adj}(t) \| $ and   \eqref{ADM}  implies that
\begin{equation}
\label{z_b}
z_{adj}(t)\in\mathbb{L}^2([s,\infty);\mathcal{H}).
\end{equation}
(See for example, \cite[Lemma A.6.5]{Curtainb}.) 

 It will now be shown that $z_{adj} (t) \to 0$ using the technique in the  proof of \cite[Lemma 4.1]{Gibson}.
Due to  \eqref{u_b} and \eqref{z_b}, for any $n$  there exists $t_n>0$ such that 
$$
\int_{t_n} ^\infty \|z_{adj}(r)\|_{\mathcal{H}}^2dr <\frac{1}{n^4}, \quad
\int_{t_n} ^\infty \|u(r)\|_{\mathcal{X}_B}^2dr < \frac{1}{n^4} . 
$$
Define $S_t=\lbrace t| t>t_n\,\&\,\|z_{adj}(t)\|_{\mathcal{H}}\geq 1/n^2\rbrace. $
Note that 
\begin{align*}
\frac{1}{n^2}\int_{S_t}dr& \leq \int_{S_t} \|z_{adj}(r)\|_{\mathcal{H}}^2dr \\
&<\frac{1}{n^4}
\end{align*}
and so   $\int_{S_t} dr <1/n^2$; write $\delta_n = \int_{S_t} dr .$
Thus, for $t\geq t_n+1/n^2$, there exists $s$ between $t_n$ and $t$ such that $|t-s| < \delta_n$ and  $\|z_{adj}(s)\|<1/n$. From \eqref{Osys}, 
\begin{align*}
\|z_{adj}(t)\|^2_{\mathcal{H}}&\leq \delta_{Y,0}\exp\left(\alpha_{Y,0}(t-s)\right)\frac{1}{n}+\delta_{Y,0}\int_s^t\exp\left(\alpha_{Y,0}(r-s)\right)\|B_c(r)\|\|u(r)\|_{}dr\\
&\leq \delta_{Y,0}\exp\left(\alpha_{Y,0}\delta_n\right)\left(1/n+\sup_t(\|B_c(t)\|)(t-s)^{1/2}\left(\int_s^\infty\|u(r)\|^2_{}dr\right)^{1/2}\right)\\
&\leq \delta_{Y,0}\exp\left(\alpha_{Y,0}\delta_n\right)\left(1/n+\sup_t(\|B_c(t)\|)1/n^3\right).
\end{align*}
Therefore,
$$\lim_{t \to \infty} \|z_{adj}(t)\|=0. $$
Theorem 4.6 in \cite{Gibson}  then implies that there is at most one solution to the Riccati integral equations \eqref{eq-Pinf} on the infinite interval and hence  $  P_{\infty}(s)$  is the optimal cost.
\end{proof}

\begin{thm}
\label{Eval-Y-ineq}
The solution $ P_{t_f}(0)$ to  \eqref{ORiccatiD} can be uniformly bounded independent of $t_f>0. $
With the  optimal  feedback  $$K_{t_f}(t)=\bd{R}_{c}^{-1}(t)\bd{B}_c^*(t) P_{t_f} (t), $$
the evolution operator $Y_{-B_c K_{t_f}}(t,s)$ defined in \eqref{Yp}  is uniformly bounded independent of $t_f .$
\end{thm}

\begin{proof}
Since $P_\infty(0)$ yields the optimal cost to the infinite-time problem, for all $t_f > 0$ and $z_{adj,0} \in \mathcal{H} , $
$$(P_{t_f}(0)z_{adj,0},z_{adj,0})_{\mathcal{H}}\leq (P_{\infty}(0)z_{adj,0},z_{adj,0})_{\mathcal{H}}$$
Thus there exists $\delta_p>0$ such that for all $t_f >0 , $
$$\|P_{t_f}(0)\|\leq \delta_p.$$
Now, from optimality of  $u_{opt}(t)=-K_{t_f}(t)z_{adj}(t)$, and definition of the cost  \eqref{Cost}  it can be concluded that for some $\delta_p^2 > 0 , $
\begin{align*}
\int_{0}^{t_f}\|W^{1/2}_c(r)z_{adj}(r)\|_{\mathcal{H}}^2dr&
\leq \delta_p^2 \|z_{adj,0}\|_{\mathcal{H}}^2\\
\int_{0}^{t_f}\|\bar{R}^{1/2}_c(r)P_{t_f}(r)z_{adj}(r)\|_{\mathcal{X}_B}^2dr&
\leq \delta_p^2 \| |z_{adj,0}\|_{\mathcal{H}}^2,
\end{align*}
and for $u(t)=u_{opt}(t)$, \eqref{Osys-r} leads to, 
 letting $M_L , $ $\alpha_L$ be such that  $\|Y_L(t,s)\|\leq M_L\exp\left(-\alpha_L(t-s)\right),$
\begin{equation*}
\begin{aligned}
\|z_{adj}(t)\|_{\mathcal{H}}&\leq
\|Y_L(t,s)\|\|z_{adj,0}\|_{\mathcal{H}}\\
&+\int_s^t\|Y_L(t,r)\|\big(\sup_{r'}\|L(r')\|\|W^{1/2}(r) z_{adj}(r)\|_{\mathcal{H}}\\
&\quad +\sup_{r'}\|\bar{R}_c^{1/2}(r')\|\| \bar{R}_c^{1/2}P_{t_f}(r)z_{adj}(r)\|_{\mathcal{U}}\big)dr\\
&\leq M_L \exp\left(-\alpha_L (t-s) \right) \|z_{adj,0}\|_{\mathcal{H}}\\
&\quad +\frac{M_L}{\sqrt{2 \alpha}}  \sup_{r'}\|L(r')\|\|W^{1/2}(t)z_{adj}(t)\|_{\mathbb{L}^2([s,\infty];\mathcal{H})}\\
&\quad +\frac{M_L}{\sqrt{2 \alpha}} \sup_{r'}\|\bar{R}_c^{1/2}(r')\|\|\bar{R}_c^{1/2}P_{t_f}(r)z_{adj}(r)\|_{\mathbb{L}^2([s,\infty];\mathcal{X}_B)}\\
& \leq  M_L \left( \exp\left(-\alpha_L(t-s) \right) +\frac{2  \delta_p}{\sqrt{2 \alpha}  }( \sup_{r'}\|L(r')\| +   \sup_{r'}\|\bar{R}_c^{1/2}(r')\|)  \right) \|z_{adj,0}\|_{\mathcal{H}}  .
\end{aligned}
\end{equation*}
Thus, 
$\|Y_{-B_cK_{t_f}}(t,s)\|$ 
is a  bounded operator independent of $t_f>0 .$
\end{proof}

The  following result now follows from Theorem \ref{thm-stab-control} and duality since  $\bd{U}_{P,\alpha}(t,s)=\Ykc^*(t_f-s,t_f-t)$ by \cite[Theorem 1.2 and 1.3]{CRiccati-1976}. 

\begin{thm}
\label{EXp_stb_evl}
Defining $\bd{A}_d(t)=\bd{A}+\alpha\bd{I}+\bd{DF}(\zh_P(t),t), $  assume that  $(\bd{A}_d(t),\bd{C})$, is  uniformly detectable and $(\bd{A}_d(t),\bd{W}^{1/2}(t))$ is uniformly stabilizable.
 Defining $\bd{L}(t)=\bd{P}(t)\bd{C}^*\bd{R}^{-1}(t)$ where. $\bd{P}(t)$ satisfies the Riccati equation \eqref{ORiccatia},  the evolution operator $\U_{P}(t,s)$ generated by $\bd{A}_d(t)-\bd{L}(t)\bd{C}$,  is uniformly bounded; 
 that is there exists $\delta_Y>0$ can be chosen so that for all $t\in[0,t_f]$
 \begin{equation}
\label{Eval-U-ineq}
\Vert\U_{P,\alpha}(t,s)\Vert_{\mathcal{H}}\leq \delta_Y.
\end{equation}
\end{thm} 

\begin{proof}

A standard duality argument  will be used to show the stability of the evolution operator generated by $\bd{A}_d-\bd{L}(t)\bd{C}$. 

Let $\bd{A}_{c}(t)=\bd{A}_d^*(t_f-t)$, $\bd{W}_{c}(t)=\bd{W}(t_f-t)$, and $\bd{R}_{c}(t)=\bd{R}(t_f-t),$ $\bd{Y}(t,s)=\bd{U}_{\alpha}^*(t_f-s,t_f-t)$, where $\bd{U}_{\alpha}(t,s)$ is generated by $\bd{A}_d(t)$; thus, $\bd{Y}(t,s)$ is a mild evolution operator. It can be concluded from \cite[Theorem 1.1 and Corollary 1.2]{CRiccati-1976} $\Ykc(t,s)=\U_{P,\alpha}^*(t_f-s,t_f-t)$ which is a mild evalution operator.  
Therefore, for $\bd{P}_{t_f}(t)=\bd{P}(t_f-t)$, the Riccati equation \eqref{ORiccatia} 
can be rewritten as  \eqref{ORiccatiD}. 
By theorem \ref{Eval-Y-ineq}, for some $\delta_Y>0$ such that
$$\Vert\Ykc(t,s)\Vert_{\mathcal{H}}\leq \delta_Y;$$
thus, by definition of $\Ykc(t,s)$, 
\eqref{Eval-U-ineq} is preserved.

\end{proof}

\begin{lem}
\label{Pre-stab}
For $\bd{D}_0(t)\in\mathbb{L}(\mathcal{H},\mathcal{H})$  continuous in time define $\bd{U}_{\bd{D}_0}(t,s)$ to be the evolution operator generated by $\bd{A}+\bd{D}_0(t)$. Furthermore, for $\beta_0>0, $ define $\bd{U}_{\bd{D}_0,\beta_0}(t,s)$ to be  the evolution operator generated by $\bd{A}+\bd{D}_0(t) +\beta_0 \bd{I} .$  Then
\begin{equation}
\label{Pre-stab-eq}
\bd{U}_{\bd{D}_0,\beta_0}(t,s)=\exp(\beta_0(t-s))\bd{U}_{\bd{D}_0}(t,s) . 
\end{equation}
\end{lem}

\begin{proof}
Note that the evolution operator $\bd{U}_{\bd{D}_0,\beta_0}(t,s)$ satisfies
\begin{equation}
\label{Pre-stab-1}
\begin{aligned}
\bd{U}_{\bd{D}_0,\beta_0}(t,s)&=\bd{T}_{\beta_0}(t-s)\\
&+\int_s^t\bd{T}_{\beta_0}(t-r)\bd{D}_0(r)\bd{U}_{\bd{D}_0,\beta_0}(r,s)dr
\end{aligned}
\end{equation}
where $\bd{T}_{\beta_0}(\cdot)$ is the semigroup generated by $\bd{A}+\beta_0\bd{I}$.
Also, from semigroup properties, it can be concluded that 
$$\bd{T}_{\beta_0}(t-s)=\exp(\beta_0(t-s))\bd{T}(t-s);$$
substituting this equality back in \eqref{Pre-stab-1} and multiplying both sides by $\exp(-\beta_0(t-s))$ result in
\begin{equation}
\label{pre-stab-2}
\begin{aligned}
&\exp(-\beta_0(t-s))\bd{U}_{\bd{D}_0,\beta_0}(t,s)=\bd{T}(t-s)\\
&+\int_s^t\bd{T}(t-r)\bd{D}_0(r)\exp(-\beta_0(r-s))\bd{U}_{\bd{D}_0,\beta_0}(r,s)dr.
\end{aligned}
\end{equation}
From definition of the evolution operator $\bd{U}_{\bd{D}_0}(t,s)$ and \eqref{pre-stab-2} it is concluded that $\bd{U}_{\bd{D}_0}(t,s)=\exp(-\beta_0(t-s))\bd{U}_{\bd{D}_0,\beta_0}(t,s)$ and the proof is complete. 
\end{proof}

\section{Error  dynamics} 

In this section, the convergence of the  estimated state $\hat{\bd z}_P (t)$ to the true state   $\bd z(t)$, with a filter $\filt (t)$ defined by    (\ref{add1a}), (\ref{add2a}), and (\ref{ORiccatia}) as $t_f \to \infty $  is  shown to hold under some additional assumptions, if the initial error is sufficiently small.

Define 
the error $\bd{e}(t)=\bd{z}(t)-\zh_P(t)$ between the system state $\bd z(t) $ and the observer state $\zh_P(t)$ and
for $\bd{z},\hat{\bd{z}}_P\in\mathcal{H}$, 
\begin{equation}
\label{phi}
\bd{\phi}( \bd e, t )=\bd{F}(\bd{z},t)-\bd F(\bd z-\bd e,t)-\bd DF(\bd z-\bd e,t)(\bd e).
\end{equation}
 From definition of the system  \eqref{Sys1} and the observer \eqref{OBD} with $\filt (t)=\bd{P}(t)\bd{C}^*\bd{R}^{-1}(t)$, 
the differential equation governing the  error dynamics is, 
\begin{equation}
\label{error-dyn}
\begin{aligned}
\frac{\partial\bd{e}(t)}{\partial t}=\bd{A}\bd{e}(t)-&\filt (t) \bd{C}\bd{e}(t)+\bd{DF}(\bd z(t)-\bd e(t),t)\bd{e}(t) \\
&+\bd{\phi}(\bd e(t))- \filt (t) \bd \eta (t) - \bd{G}\bd{w}(t).
\end{aligned}
\end{equation}
Let  $\U_{P}(t,s)$ be the evolution operator generated by $\bd{A}+\bd{DF}(\zhP(t),t)-\filt (t)\bd{C}$ where $\filt (t)=\bd{P}(t)C^*\bd{R}^{-1}(t)$ and $\bd{P}(t)$ solves \eqref{add1a}-\eqref{ORiccatia}.
The  mild solution for the error dynamics is, with initial condition $\bd e(0)=\bd e_0,$
\begin{equation}
\label{error-dyn-evl}
\begin{aligned}
\bd{e}(t)=\bd{U}_{P}(t,0)\bd{e}_0
&+\int_{0}^t\bd{U}_{P}(t,r)\bd{\phi}(\bd e(r))dr\\
&-\int_{0}^t \bd{U}_{P}(t,r)(- \bd \filt (r) \bd \eta (r) - \bd{G}\bd{w}(r))dr
\end{aligned}
\end{equation}
The right-hand side of  \eqref{error-dyn-evl} defines a nonlinear mapping 
\begin{align*}
\Phi_e(t,0, \bd e_0 )&:\mathbb{L}^2([0,t_f];\mathcal{H})\rightarrow\mathbb{L}^2([0,t_f];\mathcal{H}). 
\end{align*}
Well-posedness of the differential equations \eqref{Sys1} and \eqref{OBD} for $z(t)$ and $\hat z_P(t) $ implies that  for every initial condition $\bd{e}(0) \in\mathcal{H}$ the above equation has a unique solution in $\mathcal{C}([0,t_f],\mathcal{H})$ for $\bd  e.$ Thus, for all $\bd e(0) \in\mathcal{H}, $ the operator $\Phi $ has a unique fixed point $\bd e \in \mathbb{L}^2([0,t_f];\mathcal{H}. $

It will be shown that under stabilizability and detectability assumptions, and in the absence of disturbances, the estimation error, defined by \eqref{error-dyn} with $\bd{L}(t)=\bd{P}(t)\bd{C}^*\bd{R}^{-1}(t)$ where $\bd{P}(t)$ satisfies the Riccati equation \eqref{ORiccatia},  converges to zero if the initial error is sufficiently small.
A second-order smoothness assumption on  the  nonlinear term $\bd{F}(.)$ is needed.

\begin{thm}
\label{Conv}
Let the system $(\bd{A}_d(t),\bd{C})$, where $\bd{A}_d(t)=\bd{A}+\alpha\bd{I}+\bd{DF}(\zh_P(t),t)$, be uniformly detectable, and $(\bd{A}_d(t),\bd{W}^{1/2}(t))$ be uniformly stabilizable. 
Assume that there exist  time $T>0$ and positive numbers $m>1$, $\epsilon_{\varphi}>0$ and $\delta_{\varphi}>0$ such that if $\Vert \bd e \Vert_{\mathcal{H}}\leq \epsilon_{\phi}, $ then defining  $\bd \phi(\cdot)$  as in  \eqref{phi},
\begin{equation}
\Vert \bd{\phi}(\bd e, t) \Vert_{\mathcal{H}}\leq \delta_{\phi} \Vert \bd e \Vert_{\mathcal{H}}^m , \quad 0 \leq t \leq T .
\label{bdd-phi}
\end{equation}

In the absence of disturbances, $\bd{w}=\bd \eta \equiv 0$,   there exists  $\epsilon, $ $\delta_{e,0}>0$, $M_e>0$, such that if $\Vert e(0)\Vert = \Vert\bd{z}(0)-\hat{\bd{z}}_P(0)\Vert_{\mathcal{H}}< \epsilon$  then for $t\geq 0 , $
$$\Vert\bd{z}(t)-\hat{\bd{z}}_P(t)\Vert_{\mathcal{H}}\leq M_e\exp(-\alpha_e(t-0))\Vert\bd{z}(0)-\hat{\bd{z}}_P(0)\Vert_{\mathcal{H}}.$$

\end{thm}
\begin{proof}
 Define, for $0<t_f \leq T, $ and $\epsilon_S >0,$ the subset  of  $ \mathbb{L}^2([0,t_f);\mathcal{H})$
 $$\mathcal{S}=\lbrace v(t)\in \mathbb{L}^\infty([0,t_f);\mathcal{H}) :\|v(t)\|_{\mathbb{L}^\infty([0,t_f);\mathcal{H})}\leq  \epsilon_S  )\rbrace .$$
It will first be shown that  for suitable choices of $\epsilon$ and $\epsilon_S $  $\Phi_e(t,0,e(0))$ maps $S$ to itself if $\Vert\bd{e(0)}\Vert_{\mathcal{H}}\leq \epsilon .$
By Lemma \ref{Pre-stab} and Theorem \ref{EXp_stb_evl}, \eqref{Eval-U-ineq}, 
 the evolution operator $\bd{U}_{P}(t,s)$ generated by $\bd{A}+\bd{DF}(\zh_P(t),t)-\bd{L}(t)\bd{C}$ satisfies
\begin{equation}
\label{Eval-Y-alp-ineq}
\Vert\bd{U}_{P}(t,s)\Vert\leq\delta_Y\exp\big(- \alpha (t-s)\big) . 
\end{equation}
From \eqref{error-dyn-evl}, if   $\|\bd e_0 \| \leq \epsilon$ and $e\in\mathcal{S},$
\begin{equation}
\label{error-dyn-ine}
\begin{aligned}
\Vert \left( \Phi_e(t,0,\bd e_0 ) \bd \right) (t) \Vert_{\mathcal{H}}&\leq \Vert\bd{U}_{P}(t,0)\Vert\Vert\bd{e}(0)\Vert_{\mathcal{H}}
+\int_{0}^t\Vert\bd{U}_{P}(t,r)\Vert\Vert\bd{\phi}(\bd{z}(r),\zh_P(r))\Vert_{\mathcal{H}} dr\\
& \leq \delta_Y \exp (-\alpha t)\|   \bd{e}(0) \|_{\mathcal{H}} +\int_0^t  \delta_Y  \delta_\phi \exp (-\alpha (t-r) )  \|  \bd{e}(r) \|_{\mathcal{H}}^{m} dr \\
&\leq  ( \delta_Y \epsilon-\frac{ \delta_Y  \delta_\phi }{\alpha}\epsilon_S^m )  \exp (-\alpha t) +\frac{ \delta_Y  \delta_\phi }{\alpha} \epsilon_S^m.
\end{aligned}
\end{equation}
Choose $\epsilon >0$ and $\epsilon_S>0$ so that
\begin{align*}
\epsilon_S^{m-1} < \frac{\alpha}{\delta_Y \delta_\phi}, \\
 \epsilon < \frac{1}{\delta_Y} \epsilon_S .
\end{align*}
Then, for all $0 \leq t \leq t_f , $
$\Phi_e(t,0,\bd e_0 )$ maps $ \mathcal S$ to itself. Thus  for all initial errors $\| \bd e(0 ) \| < \epsilon ,$ if
$$\epsilon <  \frac{1}{\delta_Y}\left( \frac{\alpha}{\delta_Y \delta_\phi} \right)^{\frac{1}{m-1} },$$
the error $\|\bd e (t) \|  $ is uniformly bounded by some number $\epsilon_S< \left( \frac{\alpha}{\delta_Y \delta_\phi} \right)^{\frac{1}{m-1}}$ for all $t.$

It will now be shown that if $\| \bd{e}(0)\| < \epsilon$, the estimation error  decays exponentially.  
Define $\tilde{\bd{e}}(t) = e^{\alpha t } \bd e (t) . $ From \eqref{error-dyn-evl} and \eqref{Eval-Y-alp-ineq},
\begin{equation*}
\begin{aligned}
\| \tilde{\bd{e} }(t)\|_{\mathcal{H}} &\leq  \delta_Y\| \tilde{\bd{e}}(0) \|_{\mathcal{H}} +\int_0^{t}  \delta_Y  \delta_\phi \| \bd{e}(r)\|_{\mathcal H}^{m-1}  \| \tilde{ \bd{e}}(r) \|_{\mathcal{H}} dr  .
\end{aligned}
\end{equation*}
By Gronwall's Inequality,
$$\| \tilde{\bd{e}} (t)\|_{\mathcal{H}} \leq \delta_Y \| \tilde{ \bd{e}} (0) \|_{\mathcal{H}}\exp ( \delta_Y \delta_\phi \epsilon_S^{m-1}  t) $$
and so
\begin{equation}
\label{cont}
\begin{aligned}
\|\bd e(t) \|_{\mathcal{H}} &  \leq \delta_Y \| \bd{e} (0) \|_{\mathcal{H}}\exp\left(\alpha \left(\frac{ \epsilon_S^{m-1} \delta_Y \delta_\phi }{\alpha} -1  \right)  t\right).
\end{aligned}
\end{equation} 
Since $\epsilon_S^{m-1} < \frac{\alpha}{\delta_Y \delta_\phi},$ the error decays exponentially.
\end{proof}
Thus, in the absence of disturbances, there is exponential convergence  of the estimation error to zero, if the initial error is small enough. As expected, with smaller initial error $\epsilon$ the exponential decay  improves. 

If disturbances are present,  the estimation error is bounded.
\begin{cor}
\label{ConvC}
Consider the same assumptions as in Theorem \ref{Conv} except that there are disturbances
$\bd w(t), \bd \eta(t) $ satisfying  for some $\delta_d >0 $ $\|\bd w(t)\|,\|\bd \eta(t)\|\leq\delta_d .$ There exists $\delta_e>0$ such that if $\|\bd{e}(0)\|\leq\delta_e$ and, the estimation error is bounded.
\end{cor}
\begin{proof}
Since the undisturbed system is locally exponentially stable around the equilibrium point $\bd e(t)=0$, it is uniformly asymptotically stable in sense of \cite[Definition 2.24]{mironchenko2020}. Furthermore, the nonlinear part of the error dynamics defined by $\bd \phi(\bd z(t),\hat{\bd z}(t))-\bd K(t)\bd \eta(t)-\bd G\bd w(t)$ is uniformly Lipschitz continuous with respect to its arguments. Therefore, by \cite[Theorem 2.25]{mironchenko2020},  there exist s$\delta_e>0$ such that for $\|\bd e(0)\|_{\mathcal{H}}\leq \delta_e$ 
the disturbed system is locally input-to-state stable, and thus the estimation error $\bd{e}(t)$ is bounded.  
\end{proof}

\section{Example: estimation in a magnetic drug delivery system}
 In the absence of external stimulation, anticancer drugs distribute in the tissue via diffusion and convection.  Drug resistance in cancer treatment is a major complication, often due to an increased interstitial flow pressure that affects the direction of drug delivery. One way of  modulating the direction of drug delivery is by infusing the therapeutic agents in magnetic nanoparticles that are guided towards and within the destination tissue via electromagnetic actuation. Distribution of the nanoparticles is critical but  typically cannot be measured with high precision. However, the average density and the center of the drug distribution are  measurable quantities. An observer will be designed to estimate  nanoparticle  distribution from these measured  outputs.

A simple in-vitro magnetic drug delivery system is considered. In the first simulation,  the system is simplified such that the assumptions of the theorem stated above is satisfied. Next, we consider a more general case and show that even for a less smooth nonlinearity than assumed, the observer leads to converging estimates. 

The distribution $c(r,s,t)$
of the magnetic nanoparticles is manipulated in a fluid environment via magnetic force generated by electromagnets. 
Details of the system can be found in  ~\cite{afshar1} and also \cite{afshar2,afshar3}; only an overview is provided here. 
The electromagnets' currents are represented by $I_1(t)$ and $I_3(t)$ in the $r-$direction and $I_2(t)$ and $I_4(t)$ in the $s-$direction. The currents of the  Helmholtz coils are denoted by $I_5(t)$ in the $r-$direction and by $I_6(t)$ in the $s-$direction.
The system model is  \cite{NPT6}
\begin{equation}
\label{Mg1}
\dot{c}(t)=-\nabla\cdot(-\mathcal{D}\nabla c(t)+\kappa c(t)V_f(t)+\gamma c\nabla(H^T(t)H(t))),
\end{equation}
on the domain of interest $\Omega=[-L_0,L_0]\times[-L_0,L_0]$,
where $\mathcal{D}$ is the diffusion coefficient, $\kappa$ is the advection coefficient, $\gamma$ is a coefficient defined by the magnetic properties and size of the nanoparticles, $H$ is the magnetization vector (which is a linear function of the currents $I_k(t)$),
and $V_f(t)$ is the flow velocity field. 
The equations are solved with homogeneous Neumann boundary conditions.
The magnetization vector is $H(t)=J_cI(t)$ where $I(t)=[I_1,\cdots,I_9]^T$ and 
$J_c$ is a matrix defined by the configuration and magnetic characteristics of the electromagnets and Helmholtz coils (For more details see \cite[Chapter 4]{afshar3}.)
Thus $\nabla (H^T(t)H(t))=Q_cI_t(t)$ where $Q_c$ is a matrix with $ij$-th component  defined by $\nabla(J_{c,i}\cdot J_{c,j})$ where $J_{c,k}$ is the $k$th column of $J_c$ and 
\begin{align*}
I_t(t)=[I_1(t)I_1(t),I_1(t)I_2(t),\ldots,I_1(t)I_6(t),
I_2(t)I_2(t),\ldots,I_2(t)I_6(t),\ldots,I_6(t)I_6(t)]^T.
\end{align*}

The behaviour of the system \eqref{Mg1} is simulated over a square $\Omega$  of size $2$cm $\times$ $2$cm. 
 The diffusion coefficient is $\mathcal{D}=1\times 10^{-8}$ $\text{m}^2/\text{s}, $ $\kappa=2.5\times 10^{-7}$ and $\gamma=6.6\times 10^{-5}$. The particles have radius $500$ nm. 

Measurements of the system~\eqref{Mg1} are defined by 
\begin{equation}
\label{out}
y(t)=\int_{-L_0}^{L_0}\int_{-L_0}^{L_0}\begin{bmatrix}
rc(r,s,t)\\
sc(r,s,t)\\
c(r,s,t)\\
\end{bmatrix}
drds .
\end{equation}
In the first simulation, $V_f(t)=[\int_0^rc^2(x,s,t)dx/c(t),\int_0^sc^2(r,y,t)dy/c(t)]^T/2$. The last term on the right hand side of \eqref{Mg1} is linearized around a fixed distribution $c_e(r,s)=c_r\exp(-(r^2+s^2)/(6.25\times 10^{-5}))$, where $c_r>0$ is a normalization factor defined by
$$c_r =\mathcal{A}(\Omega)/\int_\Omega \exp(-(r^2+s^2)/(6.25\times 10^{-5}))drds$$ 
where $\mathcal{A}(\Omega)$ is the area of the domain $\Omega$.  
 The governing equations now are
\begin{equation}
\label{Mg10}
\begin{aligned}
\dot{c}(t)&=-\nabla\cdot(-\mathcal{D}\nabla c(t)+\gamma c_eQ_cI_t(t))-\kappa c^2(t)+\omega(t) .
\end{aligned}
\end{equation}
We choose the external input to be 
$$I_t(t)=[0_{1\times 9},0.8\sin(20t)/20,0_{1\times 4},16\sin(40t)/40,0_{1\times 30}]^T.$$ 
The initial  concentration is set to be uniform,   $c(\cdot,\cdot,0)\equiv 1,$ 
and the initial condition of the estimator is set at zero. 

The state space  $\mathcal{H}=\mathcal{L}^2(\Omega)$ and the state is  $z(t)=c(t)\in L^2(\Omega). $  Also,
\begin{align*}
Az(t)&=\nabla\cdot(\mathcal{D}\nabla z(t)),\\
F(z(t))&=-\kappa z^2(t),\\
Bu(t)&=\gamma c_e\nabla\cdot Q_c u(t),\quad u(t)=I_t(t).
\end{align*}
 The operator $A$ generates a $C_0$-semigroup. The nonlinear term in \eqref{Mg10}, $F(z(t))$ is Lipchitz continuous and Fr\'echet differentiable but the derivative is not uniformly bounded over the entire state space. However, if the concentration $z$ remains bounded then this function can be replaced by a uniformly bounded function in the analysis.

Both system and observer dynamics are approximated using the finite element method with square elements and piecewise linear basis functions. The order of approximation for the system is $35\times 35$. Four different orders of approximation were used for observer dynamics, $25\times 25$, $18\times 18$, $9\times 9$, and $7\times 7$. The equations are solved in MATLAB 2018. The filtering parameters are chosen to be $\alpha=8$, $W(s)=I_{\mathcal{H}}$, and $R(s)=100I_{\mathcal{Y}}$ where $I_{\mathcal{H}}$ and $I_{\mathcal{Y}}$ are the identity operators on $\mathcal{H}$ and $\mathcal{Y}$ respectively. 
The $L^2(\Omega)$-norm of the state estimation error  and also the Euclidean norm  of the error in the predicted measurement are shown in Figure \ref{F2}.
The errors converge quickly to small values. As expected, the steady-state estimation error increases as the  order of the observer  decreases.

Simulation results in the presence of system and  measurement disturbance are shown in Figures \ref{F3}. 
The disturbance signal $\omega(t)\in\mathbb{R}$ is generated via MATLAB random signal generator such that it has zero mean and covariance of $0.1$. The output disturbance $\eta(t)\in\mathbb{R}^3$ is similarly created with zero mean and a covariance matrix of $\diag(5\times 10^{-3}I_{\mathbb{R}^2},5\times 10^{-4})$.  The state estimation error increases slightly but remains bounded.  

To investigate a less smooth nonlinearity,  the velocity field is defined as a function of distribution $V_f(t)=\mathds{1}_{2\times 1}c(t)$.
The system \eqref{Mg1} is reformulated as 
\begin{equation}
\label{Mg2}
\begin{aligned}
\dot{c}(t)&=-\nabla\cdot(-\mathcal{D}\nabla c(t)+\kappa c^2(t)\mathds{1}_{2\times 1}+\kappa cQ_cI_t(t))+\omega_1(t),\\
\dot{I}_t(t)&=U_t(t)+\omega_2(t) . 
\end{aligned}
\end{equation}
In this new form,  $I_t(t)$ is an augmented state so that the state representation of the system follows the form given by \eqref{Sys1}. Here, $U_t(t)$ is the derivative of current vector and is the input to the system; we chose $U_t=[0_{1\times 9},0.8\cos(20t),0_{1\times 4},16cos(40t),0_{1\times 30}]^T$. 
The state space is now $\mathcal{H}=\mathcal{L}^2(\Omega)\times\mathbb{R}^m$ where $m>0$ is the dimension of vector $I_t$.
The state is  $z(t)=(z_1(t),z_2(t))=(c(t),I_t(t))\in L^2(\Omega_s)\times \mathbb{R}^{m}$. Also,
\begin{align*}
Az(t)&=\begin{bmatrix}
\nabla\cdot(\mathcal{D}\nabla z_1(t))&0\\
0&0
\end{bmatrix} ,
\\
F(z(t))&=-\nabla\cdot(\kappa z_1^2(t)\mathds{1}_{2\times 1}+\gamma z_1(t)Q_cz_2(t)), \\
Bu(t)&=[0,I]^T u(t),\quad u(t)=U_t(t).
\end{align*}
The nonlinear function $F(c,I):H^1(\Omega)\times\mathbb{R}^m\rightarrow \mathcal{H}$ is not Lipschitz continuous over the state space $ \mathcal{H}.$

The initial conditions and other details are the same as in the previous example. The  error in both the state estimation  and  in the measurement are shown in Figure \ref{F3}. Although the assumption of  Theorem  \ref{well-posed_thm_coupling} on the system nonlinearity is not satisfied, the estimate converges to the true state in these simulations for large-order estimator, and remains small for the lower-order estimators. 

In Figures \ref{F4}, the observer performance in presence of system and output disturbance  is shown. The disturbance signal $\omega_1(t)\in\mathbb{R}$ and $\omega_2(t)\in\mathbb{R}^4$ are generated via MATLAB random signal generator such that $\omega(t)=[\omega_1(t),\omega_2^T(t)]^T$ has zero mean and a covariance matrix of $0.1I_{\mathbb{R}^5}$. The output disturbance $\eta(t)\in\mathbb{R}^3$ is similarly created with zero mean and a covariance matrix of $\diag(5\times 10^{-3}I_{\mathbb{R}^2},5\times 10^{-4})$. Note that the generated disturbance signals are piece-wise continuous. The estimation error increases compared with the condition with no disturbances. This error increases significantly  for lower orders of the observer.

\begin{figure}
\begin{tabular}{c c}
\subfloat[]{\includegraphics[width=80mm]{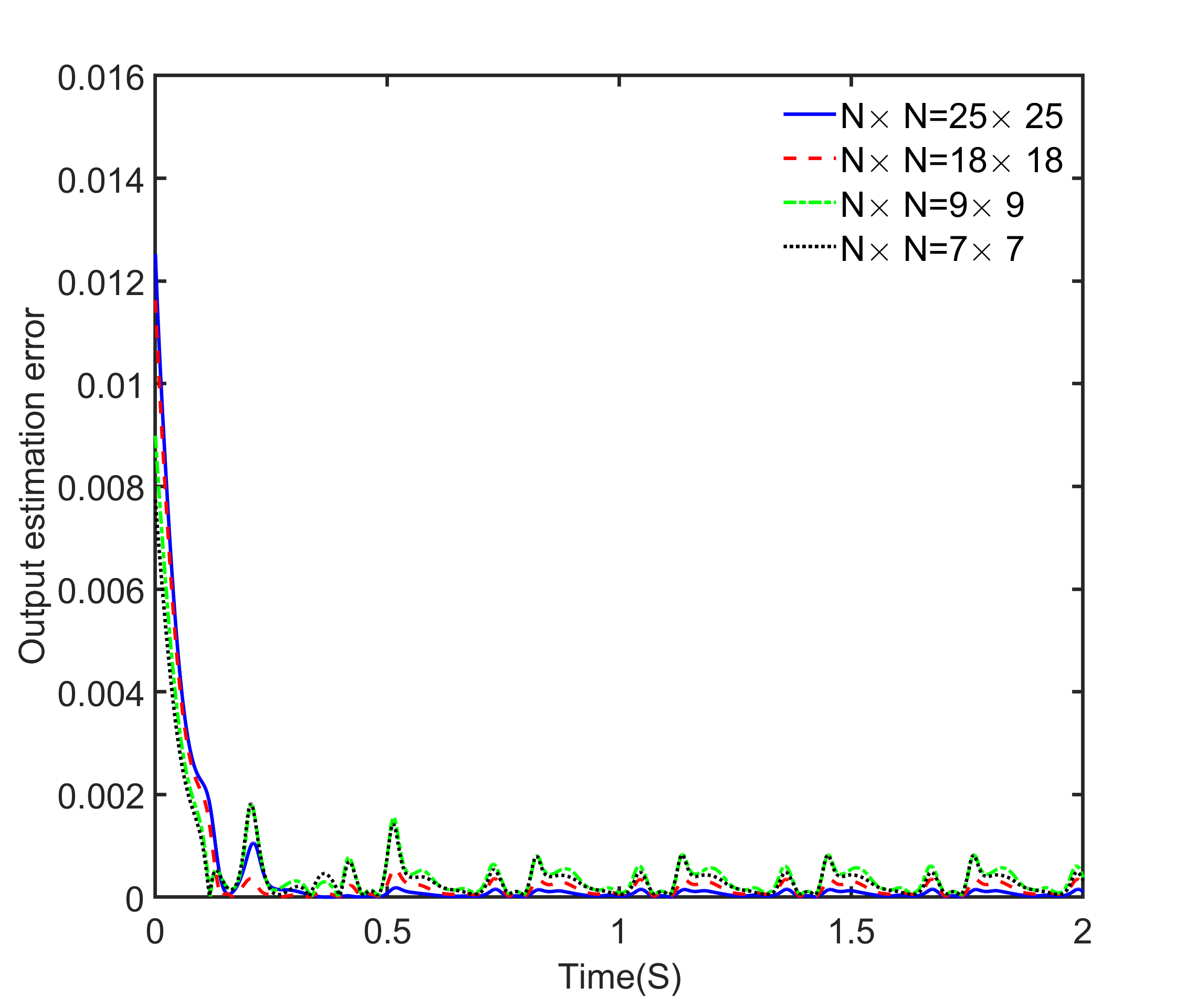}}&
\subfloat[]{\includegraphics[width=80mm]{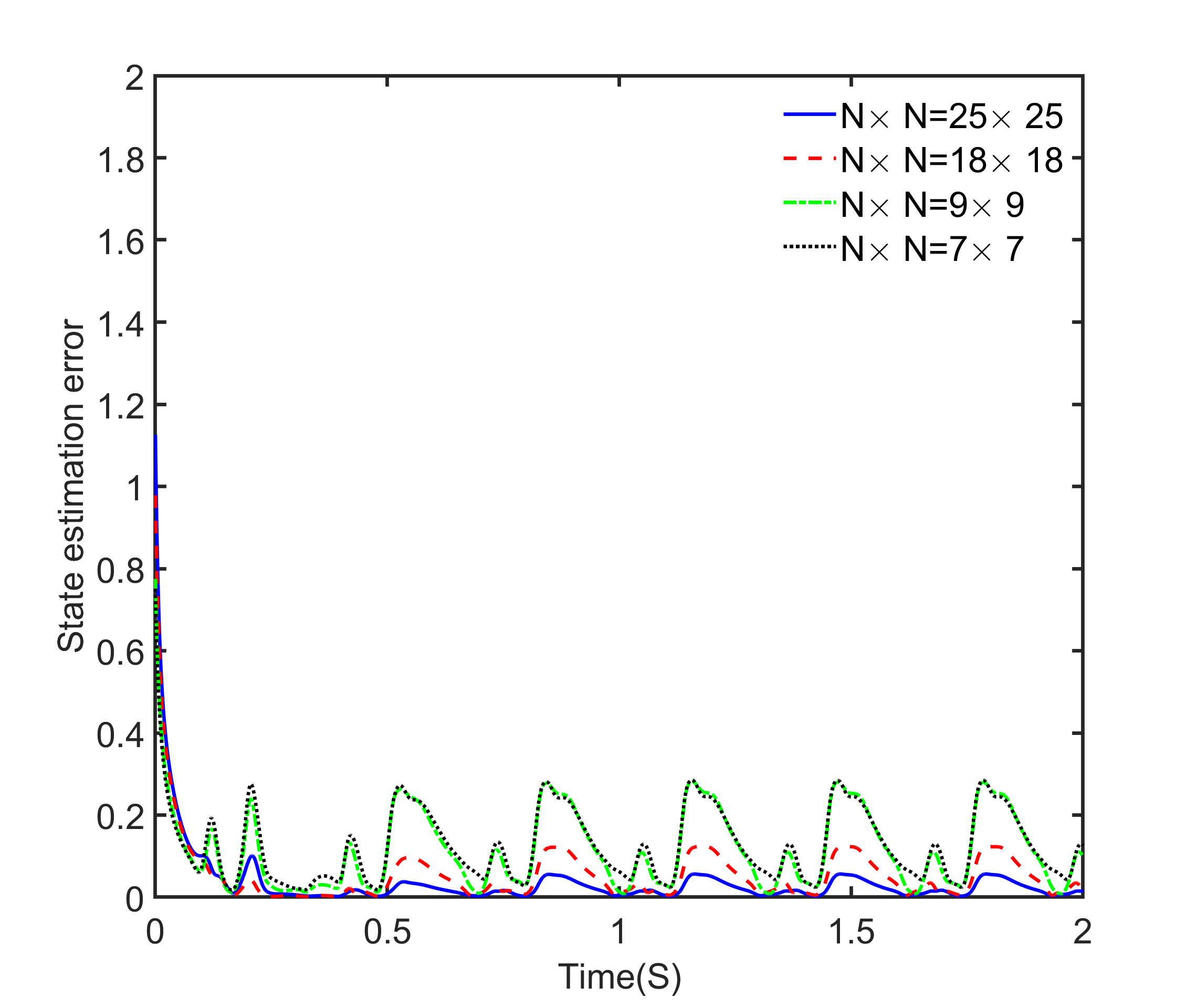}}
\end{tabular}
\caption{Error for different orders of approximation of the  observer dynamics (a) ($L^2(\Omega)$-norm of state estimation error  and (b)  Euclidean norm of output estimation error. $V_f(t)=[\int_0^rc^2(x,s,t)dx/c(t),\int_0^sc^2(r,y,t)dy/c(t)]^T/2$ and $\hat{z}=[0,0]^T$.}
\label{F1}
\end{figure}

\begin{figure}
\begin{tabular}{c c}
\subfloat[]{\includegraphics[width=80mm]{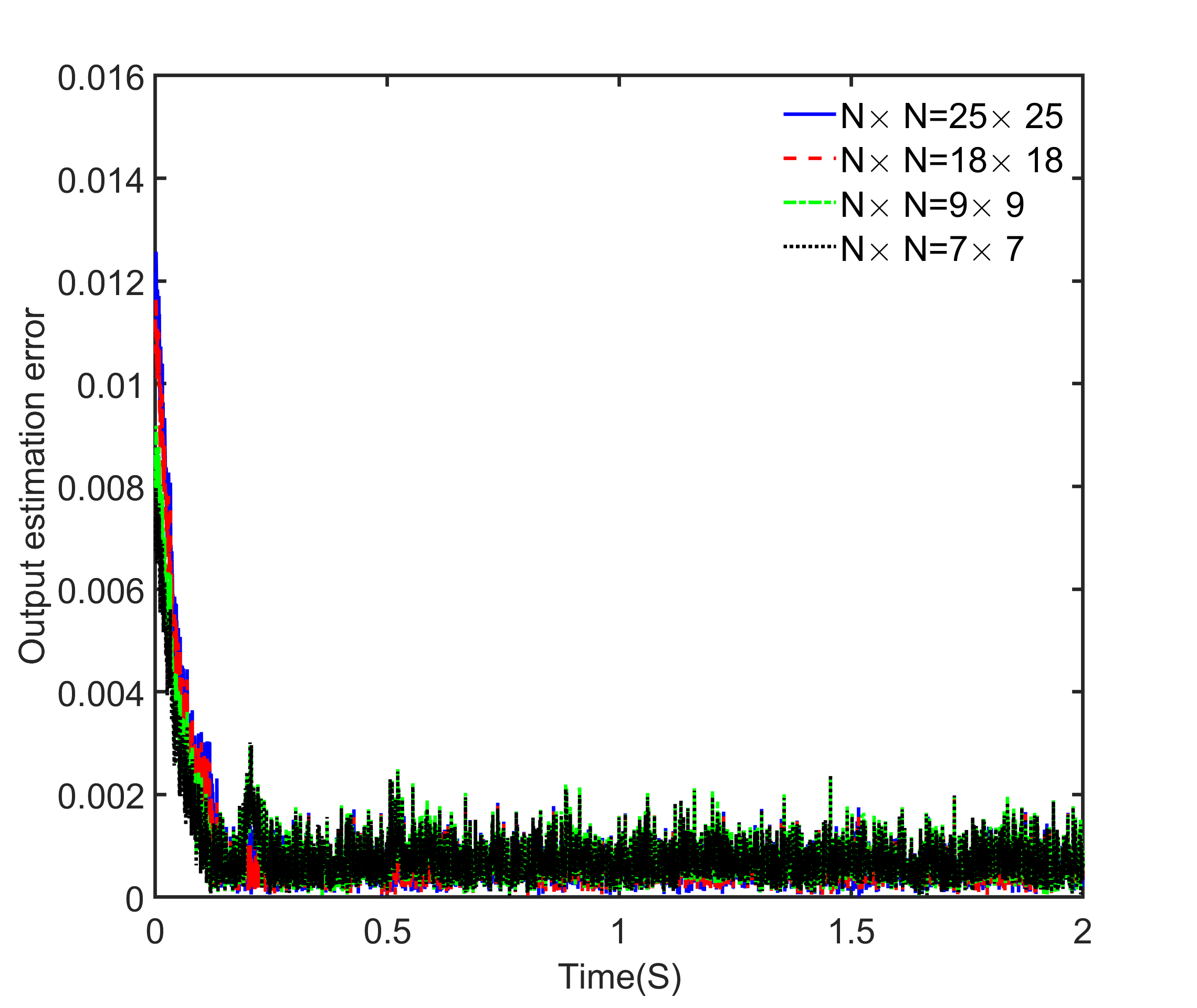}}&
\subfloat[]{\includegraphics[width=80mm]{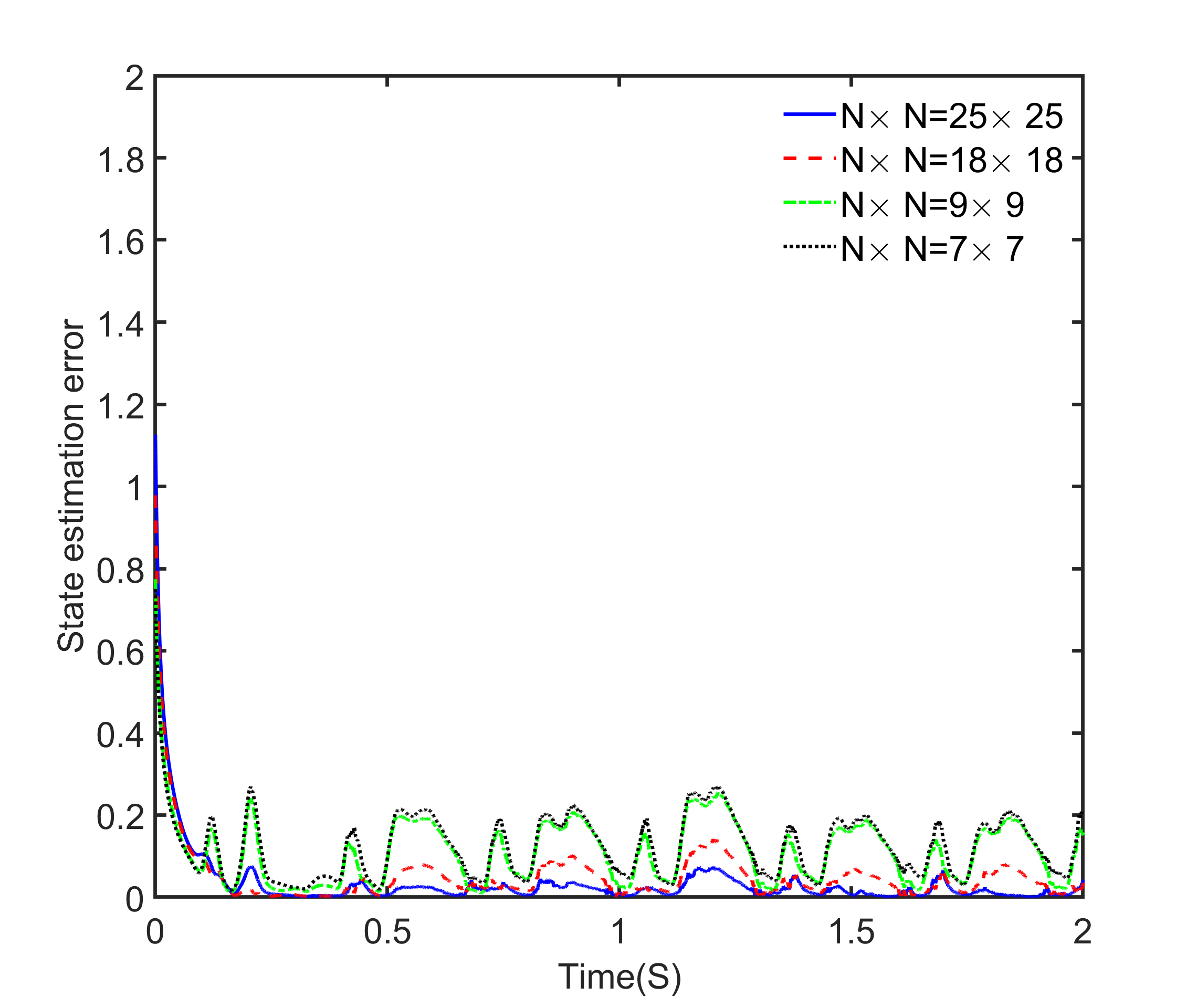}}
\end{tabular}
\caption{Error for different orders of approximation of the  observer dynamics (a) ($L^2(\Omega)$-norm of state estimation error  and (b)  Euclidean norm of output estimation error. $V_f(t)=[\int_0^rc^2(x,s,t)dx/c(t),\int_0^sc^2(r,y,t)dy/c(t)]^T/2$ and $\hat{z}=[0,0]^T$ in presence of system and output disturbance.}
\label{F3}
\end{figure}

\begin{figure}
\begin{tabular}{c c}
\subfloat[]{\includegraphics[width=80mm]{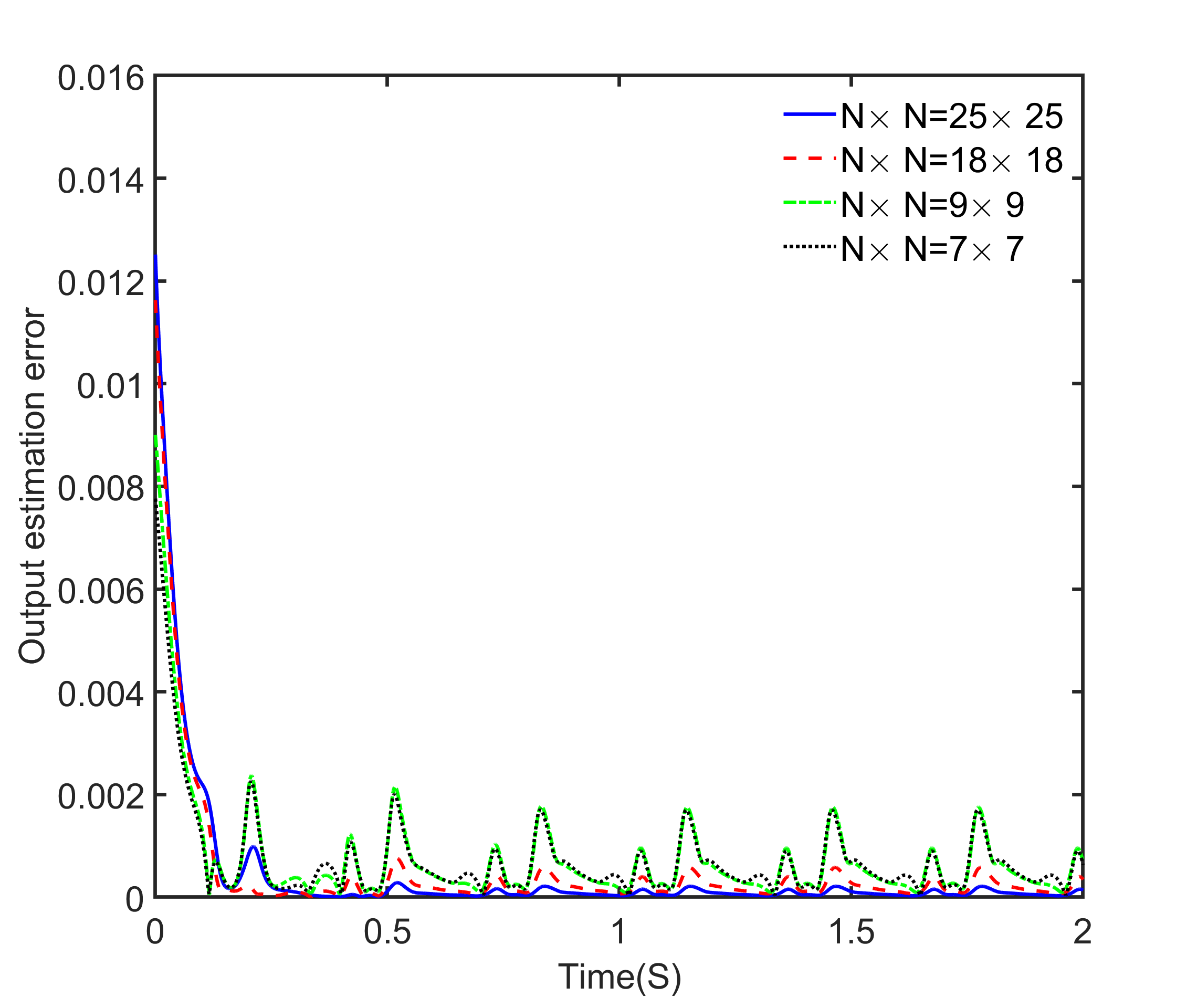}}&
\subfloat[]{\includegraphics[width=80mm]{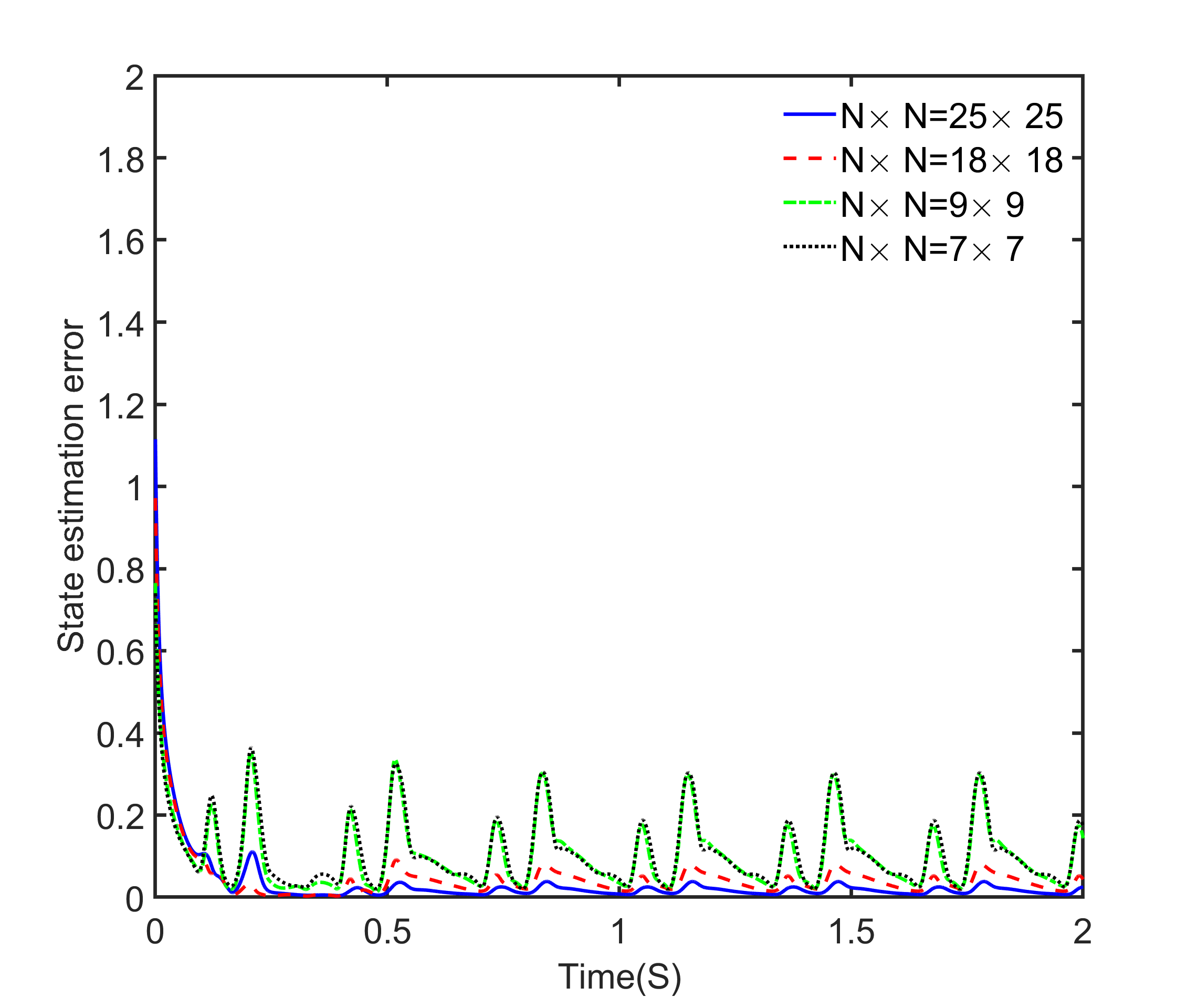}}
\end{tabular}
\caption{Error for different orders of approximation of the  observer dynamics (a) ($L^2(\Omega)$-norm of state estimation error  and (b)  Euclidean norm of output estimation error. $V_f(t)=\mathds{1}_{2\times 1}c(t)$ and $\hat{z}=0$.}
\label{F2}
\end{figure}

\begin{figure}
\begin{tabular}{c c}
\subfloat[]{\includegraphics[width=80mm]{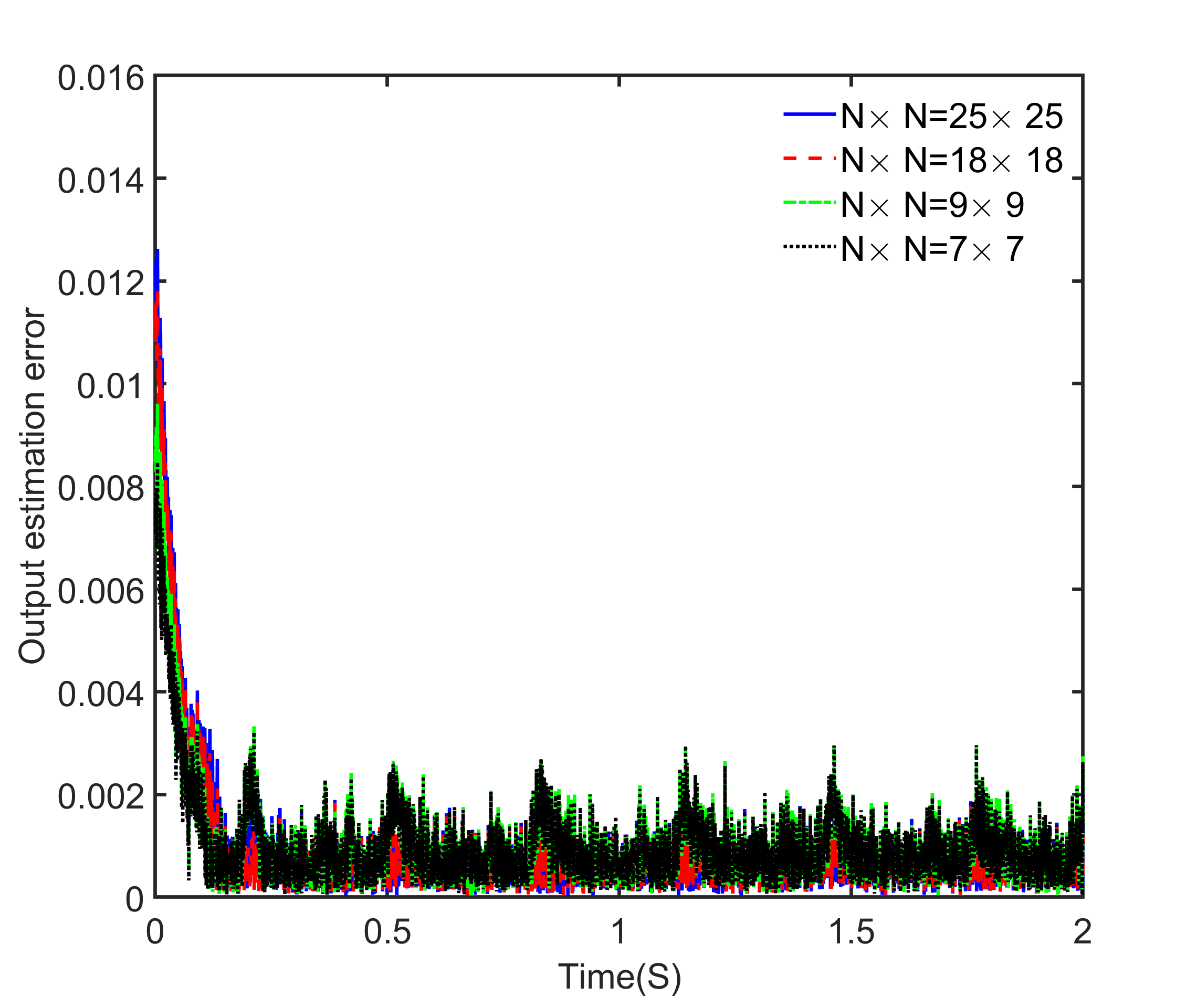}}&
\subfloat[]{\includegraphics[width=80mm]{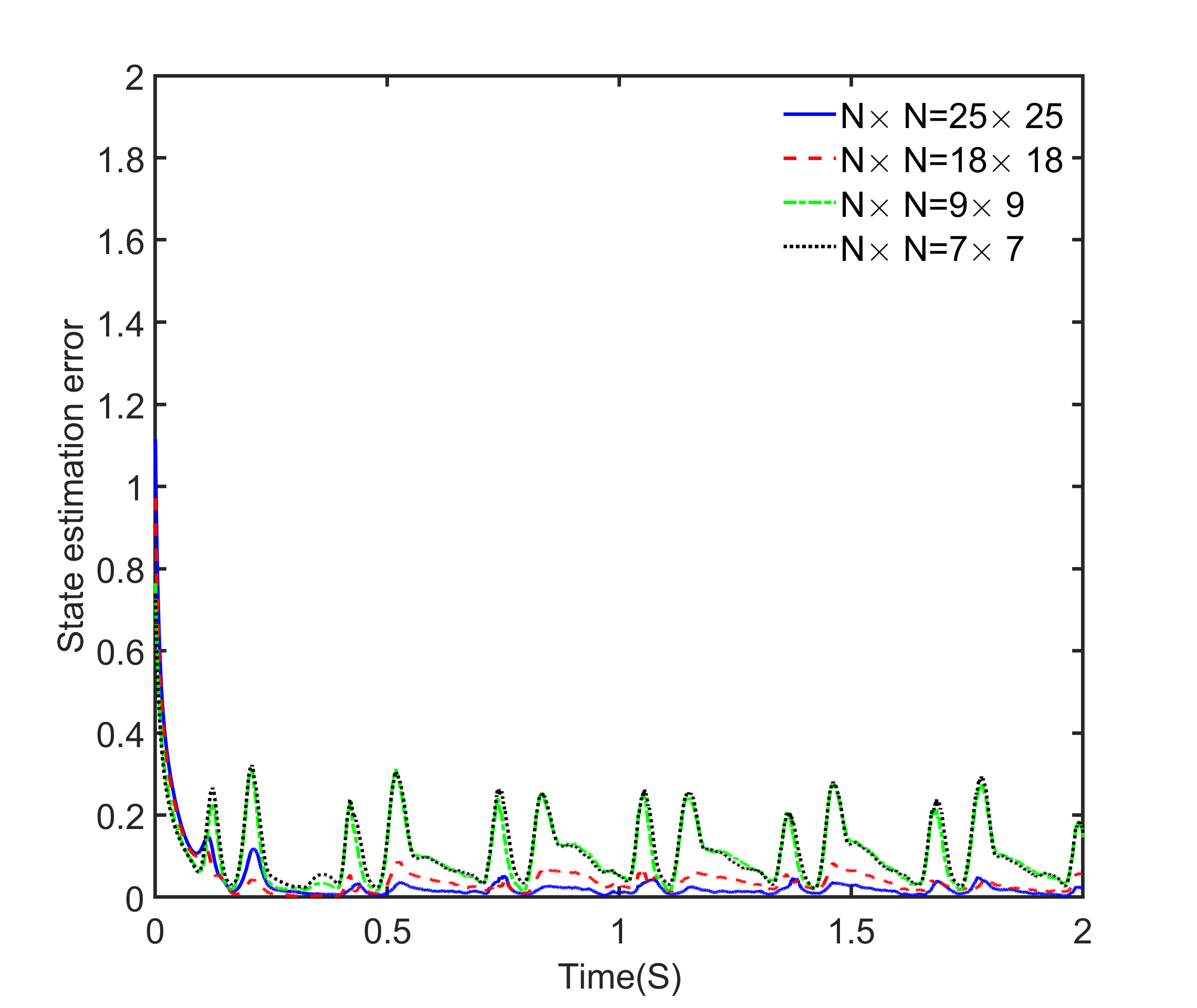}}
\end{tabular}
\caption{Error for different orders of approximation of the  observer dynamics (a) ($L^2(\Omega)$-norm of state estimation error  and (b)  Euclidean norm of output estimation error. $V_f(t)=\mathds{1}_{2\times 1}c(t)$ and $\hat{z}=0$ in presence of system and output  disturbance.}
\label{F4}
\end{figure}

\section{Conclusion}
In this paper,  EKF-based observer design for nonlinear finite-dimensional dynamical systems was  formally extended to a class of semilinear infinite-dimensional systems. It is not assumed that the system is uniformly observable, so this result is also new for finite-dimensional systems. 

The result is illustrated with several  examples.   In the  second example, the nonlinearity is not  Lipschitz continuous, but the error appears to be converging to zero.   This suggests an extension to a wider class of nonlinearities.  Also, since the estimators were of lower order than the approximation used to simulate the system, convergence of finite-dimensional estimators to the infinite-dimensional EKF is suggested.

\bibliographystyle{IEEEtran}
\bibliography{AfsharGermMorrisEKF2022}
\end{document}